\documentclass[11pt,a4paper]{amsart}
\usepackage{amsmath}
\usepackage{amsfonts}
\usepackage{graphicx}
\usepackage{framed}
\usepackage{amssymb}
\usepackage{mathtools}
\mathtoolsset{showonlyrefs}

\usepackage[dvipsnames]{xcolor}
\definecolor{cornellred}{rgb}{0.7, 0.11, 0.11}

\usepackage[colorlinks=true,urlcolor=RoyalBlue,citecolor=RoyalBlue,linkcolor=RoyalBlue,linktocpage,pdfpagelabels,
bookmarksnumbered,bookmarksopen]{hyperref}

\usepackage{hyperref}

\def \N {\mathbb{N}}
\def \R {\mathbb{R}}
\def \C {\mathbb{C}}

\def\ee {\rm e}

\def \e {\varepsilon}

\def \LL {\mathcal{L}_{\alpha}}

\usepackage{amsthm}
\usepackage{graphicx}
\usepackage{caption}
\usepackage{subcaption}
\theoremstyle{definition}
\newtheorem{definition}{Definition}[section]

\newtheorem{remark}[definition]{Remark}

\theoremstyle{plain}
\newtheorem{theorem}[definition]{Theorem}
\newtheorem{proposition}[definition]{Proposition}
\newtheorem{lemma}[definition]{Lemma}

\makeatletter
\@namedef{subjclassname@2020}{\textup{2020} Mathematics Subject Classification}
\makeatother

\numberwithin{equation}{section}
\usepackage{tikz}

\begin{document}
\title[Klein-Gordon-Maxwell equations]{Klein-Gordon-Maxwell equations driven \\
by mixed local-nonlocal operators}

\author[N.\,Cangiotti]{Nicol\`o Cangiotti}
\author[M.\,Caponi]{Maicol Caponi}
\author[A.\,Maione]{Alberto Maione}
\author[E.\,Vitillaro]{Enzo Vitillaro}

\address[N.\,Cangiotti]{Department of Mathematics\newline\indent Politecnico di Milano \newline\indent
via Bonardi 9, Campus Leonardo, 20133 Milan, Italy}
\email{nicolo.cangiotti@polimi.it}

\address[M.\,Caponi]{Dipartimento di Matematica e Applicazioni ``R. Caccioppoli''\newline\indent Università degli Studi di Napoli ``Federico II'' \newline\indent
Via Cintia, Monte S. Angelo, 80126 Naples, Italy}
\email{maicol.caponi@unina.it}

\address[A.\,Maione]{Abteilung f\"{u}r Angewandte Mathematik\newline\indent Albert-Ludwigs-
Universit\"{a}t Freiburg \newline\indent Hermann-Herder-Strasse 10, 79104 Freiburg i. Br., Germany}
\email{alberto.maione@mathematik.uni-freiburg.de}

\address[E.\,Vitillaro]{Dipartimento di Matematica e Informatica DMI\newline\indent Università degli Studi di Perugia\newline\indent
Via Luigi Vanvitelli 1, 06123 Perugia, Italy}
\email{enzo.vitillaro@unipg.it}

\subjclass[2020]{35A01, 35A15, 35J50, 35J60, 35Q60, 35R11, 58E40}

\keywords{Nonlocal operators, fractional operators, variational methods, critical points theory, Klein-Gordon-Maxwell system}


\begin{abstract}
Classical results concerning Klein-Gordon-Maxwell type systems are shortly reviewed and generalized to the setting of mixed local-nonlocal operators, where the nonlocal one is allowed to be nonpositive definite according to a real parameter.
In this paper, we provide a range of parameter values to ensure the existence of solitary (standing) waves, obtained as Mountain Pass critical points for the associated energy functionals in two different settings, by considering two different classes of potentials: constant potentials and continuous, bounded from below, and coercive potentials.
\end{abstract}

\maketitle



\section{Introduction and main results}



\medskip

In this paper we shall deal with generalized Klein-Gordon-Maxwell (KGM) type systems of the form
\begin{equation}\label{eq:problem0}
\begin{cases}
\LL u + \left[V-(\omega+e\varphi)^2\right]u = |u|^{p-2}u &\textrm{in } \R^3,\\
\Delta \varphi= e(\omega+e\varphi) u^2 &\textrm{in } \R^{3},
\end{cases}
\end{equation}
where $\omega\in\mathbb{R}\setminus\{0\}$, $e\in\{\pm 1\}$, $V\in C(\R^3)$ is bounded from below, and $p\in(2,2^*)$. Here $2^*$ denotes the classical Sobolev critical exponent $2^*=\dfrac{2n}{n-2}$ in dimension $n=3$, that is $2^*=6$.
The operator $\LL$ is a mixed local-nonlocal one of the following form
\begin{equation}\label{Lalfa}
\LL = \LL^s := -\Delta +\alpha (-\Delta)^s,
\end{equation}
where $\alpha \in \R$, $\Delta$ denotes the classical Laplacian, and $(-\Delta)^s$, $s\in (0,1)$, denotes the fractional Laplacian, which we shall introduce in the sequel.


\medskip

In the last two decades there was a growing interest around the KGM systems.
In~\cite{BF2,BF}, Benci and Fortunato introduced a mathematical model describing nonlinear Klein-Gordon fields interacting with the electromagnetic field and proved the existence of infinitely many radially symmetric solutions of~\eqref{eq:problem0} (when $\alpha=0$, that is $\LL=-\Delta$) only for $4<p<6$, by using an equivariant version of the Mountain Pass Theorem~\cite{AR,BBF,Rabinowitz86}.

The case $2<p\leq 4$ represents a more intriguing challenge, due to a lack of compactness of Palais-Smale (PS) sequences, and the extension of~\cite[Theorem 1]{BF2} and~\cite[Theorem 1.2]{BF} was later achieved by
D'Aprile and Mugnai in~\cite{Dap1}. The authors overcame the lack of compactness by requiring a control on $\omega$ by the potential $V=m^2$, with $|m|>|\omega|$, times a function depending only on $p$ (when $4<p<6$ this condition leads to the case considered by Benci and Fortunato).

A few years later, Azzolini, Pisani, and Pomponio~\cite{Azzolini, Azzolini2}, continuing along the path laid out by Benci and Fortunato, proved in the electrostatic case the existence of a ground state solution for the nonlinear Klein-Gordon-Maxwell system, refining the relation between $\omega$ and $V$, introduced in~\cite{Dap1}, and studying the limit case when the frequency of the standing wave equals the mass of the charged field.

The range $p\in(2,6)$ is not  random neither restrictive, as shown by D'Aprile and Mugnai in~\cite{Dap2} when $\alpha=0$. They proved nonexistence results based on a suitable Pohožaev identity and showed that whenever $p\leq 2$ or $p\geq 6$, $u=\varphi=0$ is the only solution to~\eqref{eq:problem0}. In~\cite{Dap1}, the authors also applied the arguments of Benci and Fortunato to the case of Schr\"odinger–Maxwell type systems. 

The critical growth case was studied by Cassani in~\cite{Cassani}, by combining a Pohožaev-type argument (to prove nonexistence of solutions with a suitable decay at infinity, as in the case for radially symmetric solutions), and the reduction method by Brezis–Nirenberg~\cite{BN83} (which allows to replace the first equation in~\eqref{eq:problem0} by adding a lower order perturbation and recover the existence of Mountain Pass type solutions).
In particular, in~\cite{Cassani}, it has been showed that whenever $|m|>|\omega|$ and $p=2^*=6$, weak solutions of~\eqref{eq:problem0} vanish identically.

In 2005, Georgiev and Visciglia~\cite{Georgiev}, inspired by the original work of Benci and Fortunato, added an external Coulomb potential in the corresponding Lagrangian density to the KGM equations and stated an existence result for these kinds of systems.
Since 2014, a renewed interest on the Klein-Gordon-Maxwell type equations with non-constant potentials (under suitable conditions) appeared on the scene, starting from the works of He~\cite{He} and Ding and Li~\cite{Ding}.

From a different perspective, a peculiar generalization is the one that involves the fractional Laplacian. Indeed, a long list of possible applications seems to be connected with fractional calculus as well it explained by Di Nezza, Palatucci, and Valdinoci in~\cite{DRV}. In this framework there is a recent and wide literature, to which this paper is inspired. Servadei and Valdinoci generalized in~\cite{Serva15} Laplace equations involving critical non-linearities of Brezis and Nirenberg~\cite{BN83}.
Before that, they also provided in~\cite{Serva12} an existence result for equations driven by a nonlocal integro-differential operator by using both fractional spaces and Mountain Pass Theorem.

Recently, a in-depth analysis of fractional KGM systems has begun as testified, on the one hand, by the work of Zhang~\cite{Zhang}, who obtained a symmetric solution for a fractional KGM system by means of variational methods and, on the other hand, by the work of Miyagaki, de Moura, and Ruviaro~\cite{MdMR19}, who found the positive ground state solution thanks again to the Mountain Pass Theorem.

The literature concerning mixed local-nonlocal operators $\LL$ is pretty vast.
As partially expected, if $\alpha \geq 0$, our existence results (Theorem~\ref{mainthm1} and Theorem~\ref{mainthm2}) are applications of the variational methods introduced in~\cite{AR,BBF,Rabinowitz86} but, differently from the case of bounded domains, we are in principle not allowed to extend trivially the study to the case $-\frac{1}{C}<\alpha<0$, where $C>0$ is the constant of the continuous embedding $H^{1}_{0}(\Omega) \subset H^{s}(\Omega)$, with $\Omega$ bounded domain of $\R^3$.

\noindent Indeed, the situation becomes suddenly more delicate, mainly because $\LL$ is no more (in general) positive definite, the bilinear form naturally associated to it does not induce a scalar product nor a norm, the variational spectrum may exhibit negative eigenvalues and even the maximum principles may fail.
It is well--known that wrong signs of parameters may change the nature of the problem considered, see for example~\cite{VV}, where a well--posed problem becomes ill--posed.

Without aim of completeness, we provide the interested reader with an overview of recent techniques aimed to face these kind of issues, mostly oriented to the (elliptic) PDEs literature.
For a very useful introduction to the variational analysis of nonlinear problem with nonlocal operators, we suggest the book of Molica Bisci, Radulescu, and Servadei~\cite{LIBRO}.

\noindent In the recent paper~\cite{MMV}, Maione, Mugnai, and Vecchi proved the existence of a weak solution of semilinear elliptic boundary value problems driven by a mixed local–nonlocal operator for every possible value of the parameter $\alpha$.
This result is obtained by means of a decomposition of the space of the solutions, deduced from the spectrum of the operator $\LL$.
An extension of this decomposition result for abstract Hilbert spaces can be found in Appendix~\ref{SectA}.

\noindent Concerning interior regularity and maximum principles, Biagi, Dipierro, Valdinoci, and Vecchi~\cite{BDVV} gave several results for elliptic operators with different orders, involving classical and fractional Laplacian.
The same authors also considered in~\cite{BDVVAcc} the qualitative properties of solutions for the same kind of mixed operators as well as the shape optimization problems~\cite{BDVV2, BDVV3}.
Moving in a similar direction, Biswas and Modasiya supplied a Faber-Krahn inequality and a one-dimensional
symmetry result related to the Gibbons’ conjecture ~\cite{BiMo}, and investigated on boundary regularity and overdetermined problems~\cite{BMS}.

\noindent Recently, De Filippis and Mingione~\cite{DeFMin} proved maximal regularity for solutions of variational mixed problems in nonlinear
degenerate cases.
Furthermore, Garain and Kinnuen~\cite{GarainKinnunen} obtained, by adopting purely analytic techniques based on the De Giorgi-Nash-Moser theory, several regularity results such as Harnack inequality for weak solutions and a weak Harnack inequality for weak supersolutions.

\noindent The relation with the mixed Sobolev inequalities was investigated by Garain and Ukhlov~\cite{GarainUkhlov}, who proved that the extremal of such inequalities, associated with an elliptic problem involving the mixed local and nonlocal Laplace operators, is unique up to a multiplicative constant.
From a different point of view, a very interesting approach, which extended the classical Bernstein technique to the setting of integro-differential operators, is due to Cabré, Dipierro, and Valdinoci~\cite{CDV22}.

\noindent Dipierro, Proietti Lippi and Valdinoci proposed a new environment in the mixed operator setting, by considering a new type of suitable Neumann conditions, with important implications to the logistic equation modeling population dynamics~\cite{DPLV1,DPLV2}.
A Brezis–Oswald approach was instead recently developed by Biagi, Mugnai, and Vecchi, leading to the full characterization of the existence of a unique positive weak solution of sublinear Dirichlet problems driven by a mixed local–nonlocal operator~\cite{BDVV5, BMV, BMV2}.

\noindent Another compelling outlook on the topic regards the asymptotic analysis performed by da Silva and Salort~\cite{SilvaSalort} and by Buccheri, da Silva, and de Miranda~\cite{BSM}.
Finally, a more exotic application can be seen in~\cite{SalortVecchi}, where Salort and Vecchi studied the existence of  the solution for Hénon-type equations driven by a nonlinear operator obtained, as before, by combining a local and a nonlocal term.

\smallskip

In this work we reformulate the original problem of Benci and Fortunato~\cite{BF2,BF}, replacing the classical Laplace operator in the Klein-Gordon equation with the mixed local-nonlocal operator $\LL$ defined in~\eqref{Lalfa}.
For the reasons stated in the previous lines, we focus in particular on the case where $\alpha$ can be negative, since negative values of $\alpha$ make the problem much more challenging.

\noindent Following the arguments in~\cite{BF2,BF}, we obtain a generalized wave equation:
\begin{equation}
    \frac{\partial^2 \phi}{\partial t^2}+\LL \phi+m^2\phi-|\phi|^{p-2}\phi=0\,.
\end{equation}
By considering stationary solutions of the form
\begin{equation}
   \phi(x,t)=u(x)\ee^{i\omega t},\quad \text{$u$ real function and }\omega \in \R\,,
\end{equation}
that are called \textit{standing waves}, we get
\begin{equation}
\LL u +(m^2-\omega^2)u=|u|^{p-2}u\,.
\end{equation}




In order to state our main existence results, we consider the case in which $\omega>0$ and $e=-1$, in which the system~\eqref{eq:problem0} reduces to
\begin{equation}\label{eq:problem}
\begin{cases}
\LL u + \left[V-(\omega-\varphi)^2\right]u = |u|^{p-2}u &\textrm{in } \R^3,\\
-\Delta \varphi= (\omega-\varphi) u^2 &\textrm{in } \R^{3}.
\end{cases}
\end{equation}
Indeed, if $(u,\varphi)$ is a solution of~\eqref{eq:problem0} for a fixed $\omega>0$ and $e=-1$, then $(u,\varphi)$ is also a solution of~\eqref{eq:problem0} with $\omega$ replaced by $-\omega$ and $e$ replaced by $-e$. Moreover $(u,-\varphi)$ is a solution of~\eqref{eq:problem0} with either $\omega$ replaced by $-\omega$ or $e$ replaced by $-e$.

In this paper we shall consider two different classes of potentials $V\colon\R^3\to\R$, namely:
\begin{itemize}
	\item[$(\text{I})$] constant positive potentials, that is $V(x)=m^2$, with $m>0$;
	\item[$(\text{II})$] continuous, bounded from below, and coercive potentials $V$, that is potentials satisfying the assumptions:
	\begin{itemize}
	    \item[$\bullet$] $V\in C(\R^3)$;
        \item[$\bullet$] $ V_0:=\inf_{x\in\R^3}V(x)>-\infty$;
        \item[$\bullet$] there exists $h>0$ such that	
        \begin{equation}\label{V}
        \lim_{|y|\to\infty}|\{x\in B_h(y)\,:\,V(x)\le M\}|=0\quad\text{for all $M>V_0$},
	    \end{equation}
which is trivially satisfied when $\displaystyle \lim_{|x|\to\infty}V(x)=\infty$.
 \end{itemize}
\end{itemize}
To handle, as far as possible, these two cases together we set a common variational framework by defining the space $\mathcal D^{1,2}(\R^3)=\overline{C_c^\infty(\R^3)}^{\|\nabla\,(\cdot)\,\|_2}$ and, for $V\in C(\R^3)$ with $V_0=\inf_{x\in\R^3}V(x)>-\infty$, also the space
\[
W=\left\{u\in H^1(\R^3)\,:\, \int_{\R^3}(V-V_0)u^2\,dx<\infty\right\}.
\]
Clearly $W$ trivially reduces to $H^1(\mathbb{R}^3)$ in the case $(\text{I})$.

At first we deal with the case $(\text{I})$, which is the nonlocal version of~\cite{Dap1}, where problem~\eqref{eq:problem} becomes
\begin{equation}\label{eq:problemBIS}
\begin{cases}
\LL u + \left[m^2-(\omega-\varphi)^2\right]u = |u|^{p-2}u &\textrm{in } \R^3,\\
-\Delta \varphi = (\omega-\varphi) u^2 &\textrm{in } \R^{3}.
\end{cases}
\end{equation}
We introduce the function $\alpha_0\colon (0,1)\times (0,\infty)\to (0,\infty)$, which is defined as
\begin{align*}
	\alpha_0(s,t):=s^{-s}(1-s)^{s-1}t^{1-s}\quad\text{for $s\in(0,1)$ and $t\in(0,\infty)$},
\end{align*}
and given $m>0$, $\omega>0$, and $p\in(2,6)$, we set
\begin{align*}
\Omega=\Omega(p,m,\omega):=m^2-\omega^2-\frac{(4-p)^+}{p-2}\omega^2.
\end{align*}
We can now state the first main result of our work.
\begin{theorem}\label{mainthm1}
In the case {\em (I)} assume that
\begin{itemize}
\item[$(a)$] when $p\in[4,6)$ we have $m>\omega>0$,
\item[$(b)$] when $p\in(2,4)$ we have $m\sqrt{p-2}>\sqrt{2}\omega>0$,
\end{itemize}
and that $\alpha>-\alpha_0(s,\Omega)$.
Then problem~\eqref{eq:problemBIS} admits infinitely many radially symmetric solutions $(u_n,\varphi_n)\in H^1(\R^3)\times\mathcal D^{1,2}(\R^3)$.
\end{theorem}

Note that under the assumptions of Theorem~\ref{mainthm1}, $\Omega\in(0,m^2)$, so that $\alpha_0(s,\Omega)$ is well defined.
\smallskip

A comparison with the classical literature is now in order.
Let us first observe that, when $\alpha=0$ and $p\in(4,6)$, Theorem
\ref{mainthm1} recovers the original results of Benci and Fortunato~\cite[Theorem 1]{BF2} and~\cite[Theorem 1.2]{BF}.
Moreover, the subsequent work of D'Aprile and Mugnai~\cite{Dap1} is also fully recovered when $\alpha=0$, in the complete range $p\in(2,6)$.
We recall that the authors proved in~\cite{Dap2} that the interval $(2,6)$ is sharp, in the sense that as long as $p\le 2$ or $p\ge 6$, the system~\eqref{eq:problemBIS} admits only the trivial solution.
Unfortunately, the generalization of this result to our context of mixed local-nonlocal operators is non--trivial. However, we feel we can conjecture that the interval $(2,6)$ is sharp even in this more general context.

\medskip

\begin{figure}[ht!]
\centering
\begin{subfigure}{.45\textwidth}
  \centering
  \includegraphics[width=1\linewidth]{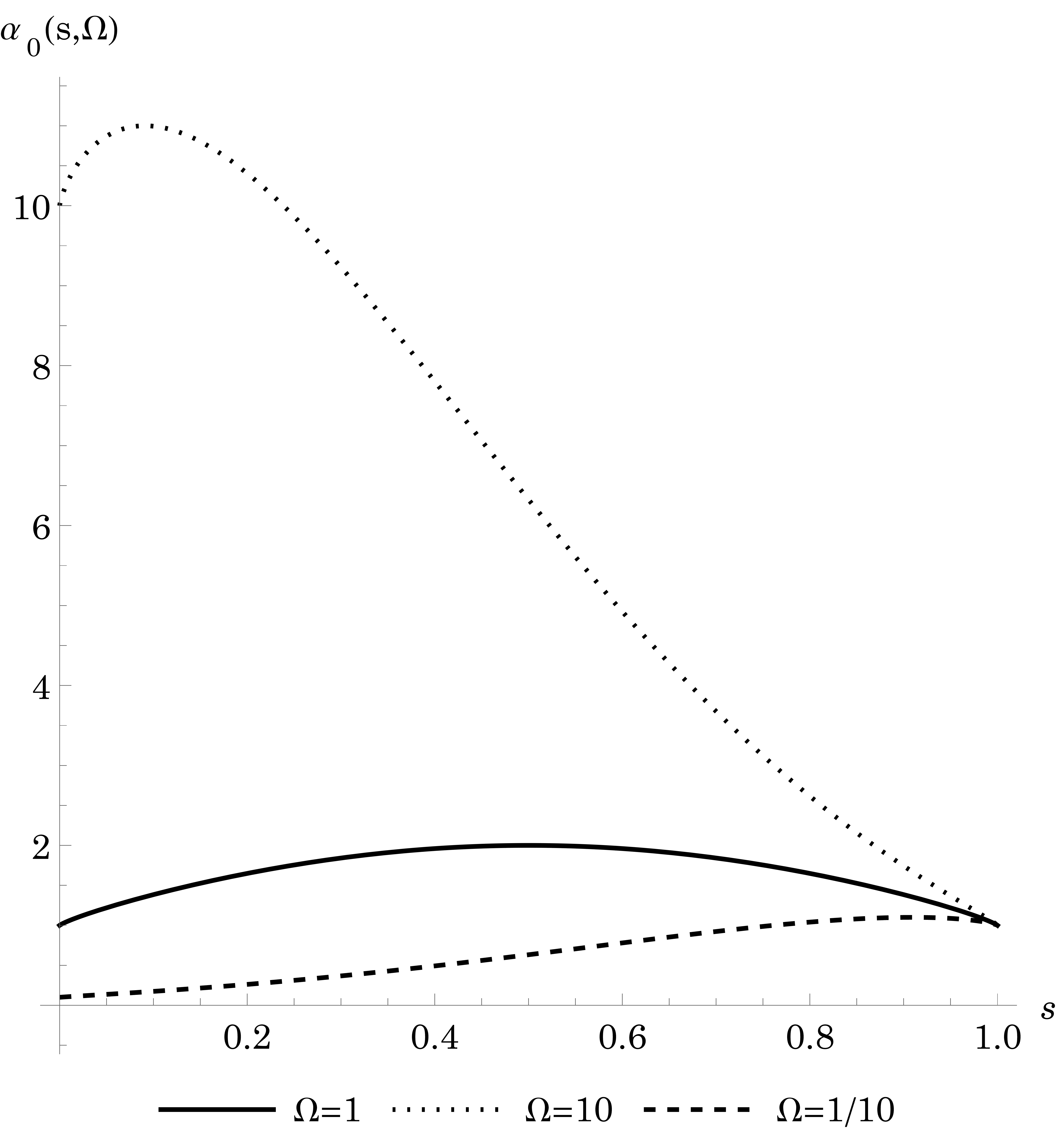}
  \caption{}
  \label{fig:sub1}
\end{subfigure}%
\begin{subfigure}{.45\textwidth}
  \centering
  \includegraphics[width=1\linewidth]{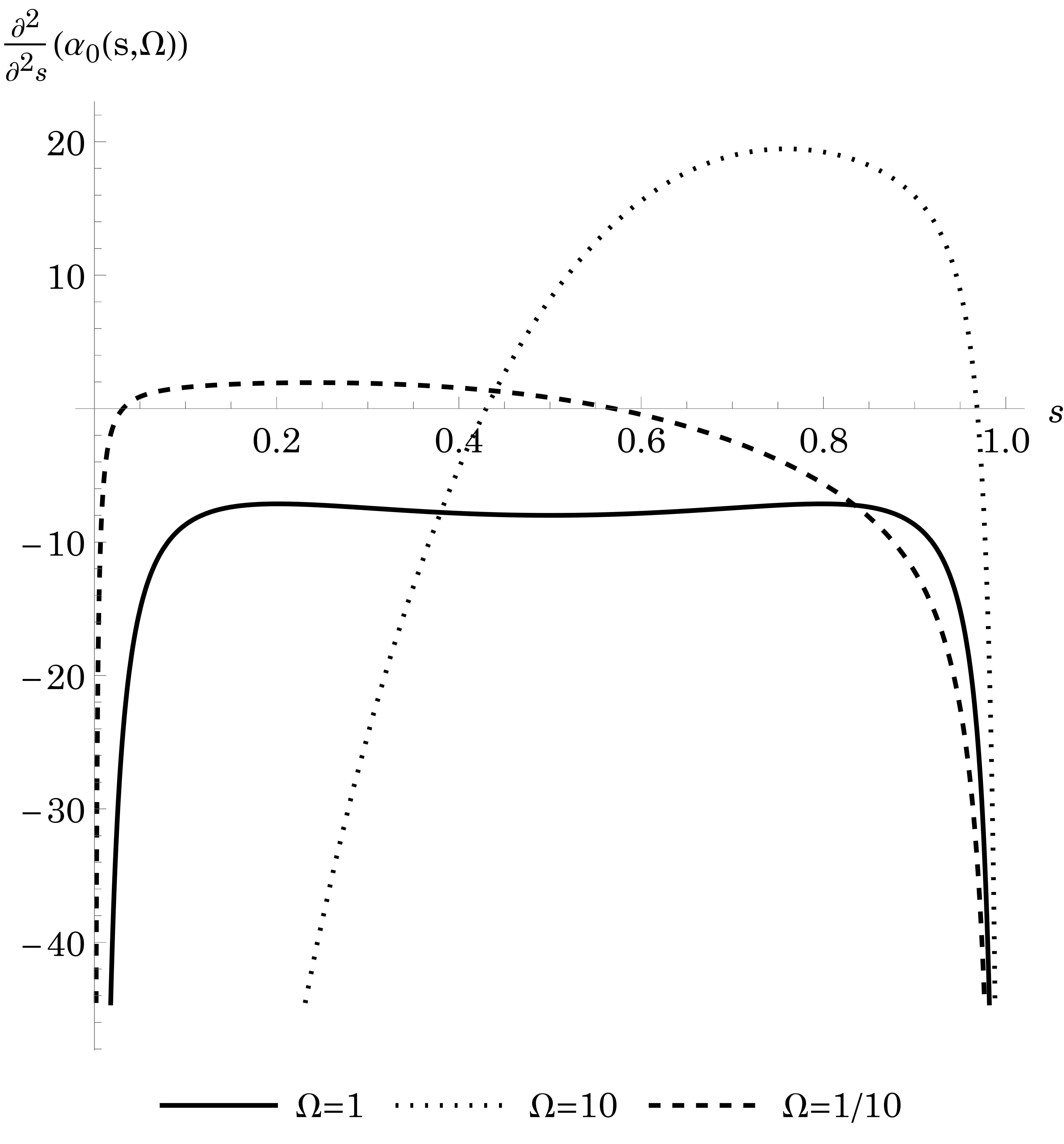}
  \caption{}
  \label{fig:sub2}
\end{subfigure}
\caption{The graph (\textsc{a}) represents the behavior of $\alpha_0(s,\Omega)$ in the interval $s\in(0,1)$, for three different fixed values of $\Omega$: $1,10,\frac{1}{10}$.
For the same values of $\Omega$, the graph (\textsc{b}) provides a representation of the second derivative of $\alpha_0(s,\Omega)$ with respect to $s$.
As one can appreciate there are two flexes for $\Omega=10$ and $\Omega=1/10$.}
\label{figure}
\end{figure}
A more in-depth study regarding the role of $\alpha_0$ is, in our opinion, important and useful to completely understand the significance of the theorem above.

\noindent As one can observe from Fig.~\ref{figure}, the limits at the boundary are
\begin{align}
    \lim_{s\to 0^+} \alpha_0(s,\Omega)&=\Omega=m^2-\omega^2-\dfrac{(4-p)^+}{p-2}\omega^2\,;\\
    \lim_{s \to 1^-} \alpha_0(s,\Omega)&= 1\,.
\end{align}
In particular, we want to underline how for $s \to 0^+$ and $s\to 1^-$ we basically recover the same results of~\cite{BF2,BF,Dap1}. On the one hand, $s=0$ formally corresponds to the operator $-\Delta u+\alpha u$. In this case, by~\cite{BF2,BF,Dap1} the system~\eqref{eq:problem} has infinitely many solutions if
\[
\alpha+m^2>\omega^2+\frac{(4-p)^+}{p-2}\omega^2,\quad\text{that is}\quad\alpha>-\Omega\,.
\]
On the other hand, $s=1$ formally corresponds to the operator $-(1+\alpha)\Delta u$, which is positively definite if and only if
\[
1+\alpha>0,\quad\text{that is}\quad\alpha>-1\,.
\]
Hence the range for $\alpha$, given by the assumption $\alpha>-\alpha_0(\Omega,s)$, seems to be sharp, at least when $s\to0^+$ and $s\to1^-$.
We conjecture that the range is sharp for all $s\in(0,1)$.
\medskip

The second main result of the paper is as follows.
\begin{theorem}
\label{mainthm2}
In the case {\em (II)} for all $p\in(2,6)$ and $\alpha\in\R$ problem~\eqref{eq:problem} admits infinitely many solutions $(u_n,\varphi_n)\in W\times\mathcal D^{1,2}(\R^3)$.
\end{theorem}
We remark that, as in the case of the Theorem~\ref{mainthm1}, Theorem~\ref{mainthm2} also recovers the classical works of Ding and Li~\cite{Ding} and He~\cite{He}, when the real parameter $\alpha$ approaches the value $0$.
We also point out that a comparison with the literature devoted to the case in which the operator $\LL$ is purely fractional (i.e. $\LL=(-\Delta)^s$) is not possible, since the parameter $\alpha$ is only coupled to the nonlocal part of the operator, while the local part of $\LL$ is fixed.

As in~\cite{BF2,BF,Dap1} and~\cite{Ding,He}, the proof of Theorems~\ref{mainthm1} and~\ref{mainthm2} are based on an equivariant version of the Mountain Pass Theorem (see~\cite{AR,BBF,Rabinowitz86}). In the forthcoming work~\cite{CCMV23}, the authors shall explore, by using variational techniques, the case of Schr\"odinger–Maxwell equation driven by mixed local-nonlocal operators.
\medskip

The paper is simply organized as follows.
Section~\ref{Sect2} is devoted to some preliminary results, which apply to both cases $(\text{I})$ and $(\text{II})$.
In Sections~\ref{Sect3} and~\ref{Sect4} we shall respectively consider the cases $(\text{I})$ and $(\text{II})$, giving the proofs of Theorems~\ref{mainthm1} and Theorem~\ref{mainthm2}.
A comprehensive overview of spectral theory for mixed local-nonlocal operators finally appears in Appendix~\ref{SectA}.


\section{Assumptions, notations, and preliminary results}\label{Sect2}

We recall that the Sobolev space $H^1(\R^3)$ is defined as
\begin{equation*}
H^1(\R^3)=\{u\in L^2(\R^3)\,:\,\nabla u\in L^2(\R^3;\R^3)\},
\end{equation*}
and it is a Hilbert space endowed with the norm
\begin{align*}
\|u\|_{H^1}^2:=\|u\|_2^2+\|\nabla u\|_2^2\quad\text{for $u\in H^1(\R^3)$}.
\end{align*}
We denote by $\mathcal F$ the Fourier transform, defined for functions $\varphi\in\mathcal S(\mathbb{R}^3)$ (the Schwartz space of rapidly decreasing smooth functions) by
\begin{equation*}
\mathcal F\varphi(\xi):=\frac{1}{(2\pi)^\frac{3}{2}}\int_{\R^3}e^{-i\xi\cdot x}\varphi(x) \,dx\quad\text{for $\xi\in\R^3$},
\end{equation*}
and then extended by density to the space of tempered distributions.
By Plancharel theorem $\mathcal F$ is an isometric isomorphism from $L^2(\R^3;\C)$ onto $L^2(\R^3;\C)$.

Given any $s\in(0,1)$, the fractional Sobolev space $H^s(\R^3)$ is equivalently defined as
\begin{equation*}
H^s(\R^3)=\left\{u\in L^2(\R^3)\,:\,\int_{\R^3}(1+|\xi|^{2s})|\mathcal Fu(\xi)|^2\,d\xi<\infty\right\},
\end{equation*}
see e.g.~\cite[Section 3]{DRV}, and it is a Hilbert space when endowed with the norm
\begin{align*}
\|u\|_{H^s}^2:=\int_{\R^3}(1+|\xi|^{2s})|\mathcal Fu(\xi)|^2\,d\xi\quad\text{for $u\in H^s(\R^3)$}.
\end{align*}
Notice that $H^1(\R^3)$ is continuously embedded into $H^s(\R^3)$ by Plancharel theorem, since for all $u\in H^1(\R^3)$
\begin{align*}
\int_{\R^3}|\xi|^{2s}|\mathcal Fu(\xi)|^2\,d\xi&\le (1-s)\int_{\R^3}|\mathcal Fu(\xi)|^2\,d\xi+s\int_{\R^3}|\xi|^2|\mathcal Fu(\xi)|^2\,d\xi\\
&=(1-s)\|u\|_2^2+s\|\nabla u\|_2^2.
\end{align*}

Let $(-\Delta)^su$ denote the fractional Laplacian of $u$, which is defined via Fourier transform for functions $\varphi\in\mathcal S(\mathbb{R}^3)$ as
\[
(-\Delta)^s\varphi(x)=\mathcal F^{-1}(|\xi|^{2s}\mathcal F\varphi(\xi))(x)\quad\text{for $x\in\R^3$}\,.
\]
By Plancherel theorem we have
\begin{equation*}
H^s(\R^3)=\{u\in L^2(\R^3)\,:\,(-\Delta)^\frac{s}{2}u\in L^2(\R^3)\}
\end{equation*}
and
\begin{equation*}
\|u\|_{H^s}^2=\|u\|_2^2+\|(-\Delta)^{\frac{s}{2}}u\|_2^2.
\end{equation*}
In particular, for all $u\in H^1(\R^3)$ and $\varepsilon>0$, we have
\begin{equation}\label{young}
\|(-\Delta)^{\frac{s}{2}}u\|_2^2=\int_{\R^3}|\xi|^{2s}|\mathcal Fu(\xi)|^2\,d\xi\le (1-s)\varepsilon^{-\frac{s}{1-s}}\|u\|_2^2+s\varepsilon\|\nabla u\|_2^2.
\end{equation}
Therefore, the fractional Laplacian can be interpreted as an operator
$$(-\Delta)^s\colon H^s(\R^3)\to H^{-s}(\R^3):=(H^s(\R^3))',$$
defined for all $u,v\in H^s(\R^3)$ as
\begin{equation}
\langle (-\Delta)^su,v\rangle_{H^{-s}(\R^3)\times H^s(\R^3)}:=\int_{\R^3}(-\Delta)^\frac{s}{2}u(-\Delta)^\frac{s}{2}v\,dx.
\end{equation}

\begin{remark}
We recall that the fractional Sobolev space $H^s(\R^3)$ can also be defined via the Gagliardo seminorm $[\,\cdot\,]_{s,2}$ as
\begin{equation}
H^s(\R^3):=\left\{u\in L^2(\R^3)\,:\,[u]_{s,2}^2:=\int_{\R^3}\int_{\R^3}\frac{|u(x)-u(x)|^2}{|x-y|^{3+2s}} \,dx \,dy<\infty\right\}.
\end{equation}
Indeed, we have
\begin{equation}
\frac{1}{2}C(s)[u]_{s,2}^2=\int_{\R^3}|\xi|^{2s}|\mathcal Fu(\xi)|^2\,d\xi\quad\text{for all $u\in H^s(\R^3)$},
\end{equation}
where the constant $C(s)$ is defined as
\begin{equation}\label{eq:Cs}
C(s):=\left(\int_{\R^3}\frac{1-\cos (x_1)}{|x|^{3+2s}} \,dx\right)^{-1}\!\!\!\!\!\!\!,
\end{equation}
see, for example~\cite[Proposition 3.4 and Proposition 3.6]{DRV}. In particular, the fractional Laplacian can be defined for $\varphi\in\mathcal S(\R^3)$ as
\begin{equation}(-\Delta)^s \varphi(x) := C(s) \,\mbox{P.V.}\int_{\R^3}\frac{\varphi(x)-\varphi(y)}{|x-y|^{3+2s}} \,dy\quad\text{for $x\in\R^3$},
\end{equation}
where P.V. denotes the Cauchy principal value, that is
\begin{equation}
\mbox{P.V.}\int_{\R^3}
\frac{u(x)-u(y)}{|x-y|^{3+2s}} \,dy:=\lim_{\varepsilon\to 0^+}\int_{\{y\in \R^3\,:\,|y-x|\ge \varepsilon\}}\frac{u(x)-u(y)}{|x-y|^{3+2s}} \,dy,
\end{equation}
and the constant $C(s)$ is the one defined by~\eqref{eq:Cs}.
\end{remark}

For all $\alpha\in\R$ we define the mixed local-nonlocal operator $\mathcal L_\alpha$ as
\begin{equation}
\mathcal L_\alpha u := -\Delta u +\alpha (-\Delta)^s u,
\end{equation}
where $\Delta u$ denotes the classical Laplace operator, while $(-\Delta)^su$ is the fractional Laplacian. As before we can interpret the mixed local-nonlocal operator $\mathcal L_\alpha$ as an operator
$$\mathcal L_\alpha\colon H^1(\R^3)\to H^{-1}(\R^3):=(H^1(\R^3))',$$
to which we can naturally associate a bilinear form as follows.

\begin{definition}\label{def:Balpha}
The bilinear form $\mathcal B_\alpha\colon H^1(\R^3)\times H^1(\R^3)\to \R$ (associated to the operator $\mathcal L_\alpha$) is defined for all $u,v\in H^1(\R^3)$ as
\begin{align*}
\mathcal B_\alpha(u,v):=&\int_{\R^3}\langle\nabla u,\nabla v\rangle \,dx+\alpha\int_{\R^3}(-\Delta)^\frac{s}{2}u\,(-\Delta)^\frac{s}{2}v\,dx\\
=&\int_{\R^3}\langle\nabla u,\nabla v\rangle \,dx+\alpha\frac{C(s)}{2}\int_{\R^3}\int_{\R^3}\frac{(u(x)-u(y))(v(x)-v(y))}{|x-y|^{3+2s}}\,dx\,dy.
\end{align*}
Clearly $\mathcal B_\alpha$ is well defined and continuous on $H^1(\R^3)\times H^1(\R^3)$.
\end{definition}

Let $V\in C(\R^3)$ be a potential with
\begin{equation*}
V_0:=\inf_{x\in\R^3}V(x)>-\infty
\end{equation*}
(clearly this is satisfied for both cases (I) and (II)). The space of solutions $u$ of problem~\eqref{eq:problem} is defined as
$$W:=\left\{u\in H^1(\R^3)\,:\,\int_{\R^3}(V-V_0)u^2\,dx<\infty\right\},$$
endowed with the norm
\[
\|u\|_W^2:=\|u\|_2^2+\|\nabla u\|_2^2+\int_{\R^3}(V-V_0)u^2\,dx\,.
\]

\begin{lemma}\label{lem:den}
$W$ is a Hilbert space with respect to $\|\cdot\|_W$. Moreover, the space $C_c^\infty(\R^3)\subset W$ is dense in $W$.
\end{lemma}

\begin{proof}
It is clear that $W\subset H^1(\R^3)$ is a linear subspace of $H^1(\R^3)$, and the map $\|\cdot\|_W\colon W\to [0,\infty)$ is a norm on $W$ which is induced by a scalar product. We need just to show that $W$ is a closed subspace of $H^1(\R^3)$. Let $(u_k)_k\subset W$ and $u\in H^1(\R^3)$ be such that $\|u_k-u\|_W\to 0$ as $k\to\infty$. Then, for a fixed $k_0\in\N$
$$(V-V_0)u^2\le 2(V-V_0)(u_{k_0}-u)^2+2(V-V_0)u_{k_0}^2\quad\text{a.e. in $\R^3$},$$
which implies that
$$\int_{\R^3}(V-V_0)u^2\,dx\le 2\int_{\R^3}(V-V_0)(u_{k_0}-u)^2\,dx+2\int_{\R^3}(V-V_0)u_{k_0}^2\,dx.$$
Hence $u\in W$, i.e., $W$ is a closed subspace of $H^1(\R^3)$.

We first notice that $C_c^\infty(\R^3)\subset W$, being $C_c^\infty(\R^3)\subset H^1(\R^3)$ and
$$\int_{\R^3}(V-V_0)u^2\,dx\le \|u\|_{L^\infty(\R^3)}^2\max_{\textrm{supp}\, u}(V-V_0)<\infty$$
for all $u\in C_c^\infty(\R^3)$. Let $u\in W$ and consider a sequence $(\chi_k)_k\subset C_c^\infty(\R^3)$ of functions satisfying
$0\le \chi_k\le 1$ in $\R^3$, $\chi_k=1$ in $ B_k(0)$, and $\chi_k=0$ in $\R^3\setminus B_{k+1}(0)$. Clearly $(u\chi_k)_k\subset W$, every $\chi_k u$ has compact support in $\R^3$, and $u\chi_k\to u$ in $W$ as $k\to\infty$. Hence for all $\varepsilon>0$ there exists $k_0\in \N$ such that
$$\|\chi_{k_0}u-u\|_W<\frac{\varepsilon}{2}.$$
Let $(\rho_j)_j\subset C_c^\infty(\R^3)$ be a sequence of mollifiers in $\R^3$. Then $(\rho_j*(\chi_{k_0}u))_j\subset C_c^\infty(\R^3)$ and $\rho_j*(\chi_{k_0}u)\to \chi_{k_0}u$ a.e. in $\R^3$ as $j\to \infty$. Hence, there exists $j_0\in\N$ such that
$$\|\rho_{j_0}*(\chi_{k_0}u)-\chi_{k_0}u\|_{H^1(\R^3)}<\frac{\varepsilon}{2(1+C_{V,k_0})},\quad C_{V,k_0}^2:=\max_{\overline B_{k_0+1}(0)}(V-V_0).$$
Hence
$$\|\rho_{j_0}*(\chi_{k_0}u)-\chi_{k_0}u\|_W\le (1+C_{V,k_0})\|\rho_{j_0}*(\chi_{k_0}u)-\chi_{k_0}u\|_{H^1(\R^3)}<\frac{\varepsilon}{2},$$
which gives
$$\|\rho_{j_0}*(\chi_{k_0}u)-u\|_W<\varepsilon.$$
Therefore $C_c^\infty(\R^3)$ is dense in $W$.
\end{proof}

Since
$$\|u\|_{H^1}\le \|u\|_W\quad\text{for all $u\in W$},$$
we derive that the embedding $W\subset L^p(\R^3)$ is continuous and dense for all $p\in[2,6]$, being $6=2^*$ the critical Sobolev exponent for $n=3$. In particular, there exists a constant $C_p>0$ such that
\begin{equation}\label{eq:Sobolev}
\|u\|_p\le C_p\|u\|_W\quad\text{for all $u\in W$}.
\end{equation}

The space of solutions for the electrical potential $\varphi$ of problem~\eqref{eq:problem} is the Hilbert space, already introduced in Section~1,
\begin{equation*}
\mathcal D^{1,2}(\R^3)=\overline{C_c^\infty(\R^3)}^{\|\nabla\,(\cdot)\,\|_2},
\end{equation*}
endowed with the norm
$$\|\varphi\|_{\mathcal D^{1,2}}:=\|\nabla\varphi\|_2\quad\text{for all $\varphi\in\mathcal D^{1,2}(\R^3)$}.$$
Since in the whole space $\R^3$ the Poincar\'e inequality does not hold, we get
\begin{equation}
\mathcal D^{1,2}(\R^3)\neq H^1_0(\R^3)=H^1(\R^3).
\end{equation}
In any case, $\mathcal D^{1,2}(\R^3)$ is continuously embedded into $L^6(\R^3)$, i.e., there exists a constant $C_D>0$ such that
$$\|\varphi\|_6\le C_D\|\varphi\|_{\mathcal D^{1,2}}\quad\text{for all $\varphi\in\mathcal D^{1,2}(\R^3)$}.$$

We can now introduce the definition of weak solutions of~\eqref{eq:problem}.
\begin{definition}\label{eq:DefWeakSol}
A pair $(u,\varphi)\in W\times\mathcal D^{1,2}(\R^3)$ is called a weak solution of~\eqref{eq:problem} if
\begin{equation}\label{eq:u}
\mathcal B_{\alpha}(u,v) + \int_{\R^3}V uv \,dx+\int_{\R^3}(\omega-\varphi)^2 uv \,dx= \int_{\R^3}|u|^{p-2}uv\,dx\quad\text{for all $v\in W$}
\end{equation}
and
\begin{equation}\label{eq:varphi}
\int_{\mathbb R^3}\langle\nabla\varphi,\nabla\psi\rangle \,dx = \int_{\R^3}(\omega-\varphi)\psi u^2\,dx\quad\text{for all $\psi\in\mathcal D^{1,2}(\R^3)$}.
\end{equation}
\end{definition}

To show that Definition~\ref{eq:DefWeakSol} makes sense we state and prove the following result.
\begin{lemma}\label{spacesolutions}
The system is coherent, whether $u,v\in W$ and $\varphi,\psi\in\mathcal D^{1,2}(\R^3)$.
\end{lemma}

\begin{proof}
Let us show that all the terms in~\eqref{eq:u} and~\eqref{eq:varphi} are well defined for $u,v\in W$ and $\varphi,\psi\in \mathcal D^{1,2}(\R^3)$. As observed before, the bilinear form $\mathcal B_\alpha$ is well defined and continuous on $W\times W\subset H^1(\R^3)\times H^1(\R^3)$. Moreover, by Hölder inequality,
\begin{align*}\left|\int_{\R^3}V uv \,dx\right|&\le \left|\int_{\R^3}(V-V_0)uv \,dx\right|+|V_0|\left|\int_{\R^3}uv \,dx\right|\\
&\le \|u\|_W\|v\|_W+|V_0|\|u\|_2\|v\|_2<\infty
\end{align*}
for every $u,v\in W$.
By the same arguments used in~\cite{BF}, we also have
\begin{align*}
\left|\int_{\R^3}(\omega-\varphi)^2 uv \,dx\right|&\le \omega^2\|u\|_2\|v\|_2+2\omega\|\varphi\|_6\|u\|_{\frac{12}{5}}\|v\|_{\frac{12}{5}}+\|\varphi\|_6^2\|u\|_3\|v\|_3<\infty,\\
\left|\int_{\R^3}|u|^{p-2}uv\,dx\right|&\le \|u\|_p^{p-1}\|v\|_p<\infty
\end{align*}
for every $u,v\in W$ and $\varphi\in\mathcal D^{1,2}(\R^3)$. On the other hand
\begin{align*}
\left|\int_{\mathbb R^3}\langle\nabla\varphi,\nabla\psi\rangle \,dx\right|&\le \|\varphi\|_{\mathcal D^{1,2}}\|\psi\|_{\mathcal D^{1,2}}<\infty,\\
\left|\int_{\R^3}(\omega-\varphi)\psi u^2\,dx\right|&\le\omega\|\psi\|_6\|u\|_{\frac{12}{5}}^2+\|\varphi\|_6\|\psi\|_6\|u\|_3^2<\infty
\end{align*}
for every $u\in W$ and $\varphi,\psi\in\mathcal D^{1,2}(\R^3)$.
\end{proof}

It is easy to see that regular solutions of~\eqref{eq:problem} are actually weak solutions, according to Definition~\ref{eq:DefWeakSol}.
As usual, weak solutions of~\eqref{eq:problem} can be found as critical points of the functional $F\colon W\times\mathcal D^{1,2}(\R^3) \to \R$, defined as
\begin{align*}
F(u,\varphi):= & \dfrac{1}{2}\mathcal B_\alpha(u,u)+\dfrac{1}{2}\int_{\R^3}Vu^2\,dx-\dfrac{1}{2}\int_{\R^3}(\omega-\varphi)^2u^2\,dx\\
&-\frac{1}{p}\int_{\R^3}|u|^p\,dx-\dfrac{1}{2}\int_{\R^3}|\nabla\varphi|^2 \,dx.
\end{align*}
As in~\cite{BF}, the functional $F$ is Fr\'{e}chet differentiable on $W\times\mathcal D^{1,2}(\R^3)$ and for all $u,v\in W$ and $\varphi,\psi\in\mathcal D^{1,2}(\R^3)$ we have
\begin{align*}
F'_u(u,\varphi)[v] &=\mathcal B_\alpha(u,v)+\int_{\R^3}Vuv \,dx-\int_{\R^3}(\omega-\varphi)^2uv \,dx-\int_{\R^3}|u|^{p-2}uv \,dx,\\
F'_\varphi(u,\varphi)[\psi] &=\int_{\R^3}(\omega-\varphi)u^2\psi \,dx-\int_{\R^3}\langle\nabla\varphi,\nabla\psi\rangle \,dx.
\end{align*}
Unfortunately, even though it seems to be natural to work with the functional $F$, we are unable to endow the Hilbert space $W\times\mathcal D^{1,2}(\R^3)$ with a norm suitable to apply the theory of critical points to $F$. Therefore, we look for another variational characterization of problem~\eqref{eq:problem}.

First, we fix $u\in W$ and we look for a solution $\varphi(u)$ of~\eqref{eq:varphi}.
Since $\varphi$ is a solution of~\eqref{eq:varphi} if and only if $-\varphi$ is a solution of~(2.2) in~\cite{Dap1}, we can restate~\cite[Proposition~2.2]{Dap1} is the following form (more suitable in the present context).
\begin{lemma}\label{lem:varphi}
For every $u\in W$ there exists a unique $\varphi(u)\in\mathcal D^{1,2}(\R^3)$ which solves~\eqref{eq:varphi}. Moreover,
$$\text{$\varphi(u)\ge 0$ in $\R^3$ and $\varphi(u)\le \omega$ on the set $\{x\in\R^3\,:\, u(x)\neq 0\}$}.$$
Finally, if $u$ is radially symmetric, then also $\varphi(u)$ is radially symmetric.
\end{lemma}



\begin{remark}
In the general case in which $\omega\in\R\setminus\{0\}$ and $e\in\{\pm 1\}$ we deduce that for every $u\in W$ the unique solution $\varphi(u)\in\mathcal D^{1,2}(\R^3)$ of~\eqref{eq:varphi} satisfies
$$-\frac{e}{\omega}\varphi(u)\ge 0\text{ in $\R^3$ and }-\frac{e}{\omega}\varphi(u)\le 1\text{ on the set $\{x\in\R^3\,:\, u(x)\neq 0\}$}.$$
\end{remark}

Fixed $u\in W$, let $\varphi_u:=\varphi(u)\in\mathcal D^{1,2}(\R^3)$ be the unique solution of~\eqref{eq:varphi}. Then, $F_\varphi'(u,\varphi_u)[\psi]=0$ for every $\psi\in\mathcal D^{1,2}(\R^3)$ and for $\psi=\varphi_u$ we get
\begin{equation}\label{eq:idvarphiu}
\int_{\mathbb R^3}|\nabla\varphi_u|^2\,dx=\int_{\R^3}(\omega-\varphi_u)\varphi_u u^2\,dx.
\end{equation}
This allows us to introduce the following functional, as done in~\cite{BF}.
\begin{definition}
Fix any function $u\in W$, let $\varphi_u\in\mathcal D^{1,2}(\R^3)$ be the unique solution of~\eqref{eq:varphi}.
We define the functional $J\colon W\to\R$ by
\begin{equation}\label{eq:J}
J(u):=\dfrac{1}{2}\mathcal B_\alpha(u,u) +\dfrac{1}{2}\int_{\R^3}(V -\omega^2)u^2\,dx+\dfrac{\omega}{2}\int_{\R^3}\varphi_u u^2\,dx-\frac{1}{p}\int_{\R^3}|u|^p \,dx.
\end{equation}
\end{definition}
By the identity~\eqref{eq:idvarphiu}, we have
\[
J(u)=F(u,\varphi_u)\,.
\]
Moreover, by standard arguments, the map $u\mapsto\varphi_u$ from $W$ into $\mathcal D^{1,2}(\R^3)$ is of class $C^1$ (for a detailed proof, we refer to~\cite[Proposition 2.1]{Dap2}). Hence, the functional $J$ is Fr\'{e}chet differentiable on $W$ and
\begin{equation*}
J'(u)[v] =F'_u(u,\varphi_u)[v]\quad\text{for all $u,v\in W$},
\end{equation*}
since $F_\varphi'(u,\varphi_u)[\varphi_u'[v]]=0$, that is
\begin{align*}
J'(u)[v]&=\mathcal B_\alpha(u,v)+\int_{\R^3}(V-\omega^2)uv \,dx+2\omega\int_{\R^3}\varphi_u uv \,dx\\
&\quad-\int_{\R^3}\varphi_u^2uv \,dx-\int_{\R^3}|u|^{p-2}uv \,dx\quad\text{for any }u,v\in W.
\end{align*}
Therefore, as in~\cite{BBF}, a pair $(u,\varphi)\in W\times\mathcal D^{1,2}(\R^3)$ is a weak solution of problem~\eqref{eq:problem} if and only if $\varphi=\varphi_u$ and $u$ is a critical points of $J$.

Hence, in order to find solutions of problem~\eqref{eq:problem} it is enough to find critical points of $J$ on $W$. This is done be applying an equivariant version of the Mountain Pass Theorem, in the form given by~\cite[Theorem 9.12]{Rabinowitz86} (see also~\cite[Theorem 2.13]{AR} and~\cite[Theorem 2.4]{BBF}). First, we recall the following definition.

\begin{definition}
Let $f$ be a $C^1$ function, defined on an infinite dimensional Banach space $X$. We say that the functional $f$ satisfies the
Palais–Smale condition $(PS)$ if any sequence $(u_n)_n\subset X$ such that $(f(u_n))_n\subset\R$ is bounded and $f'(u_n)\to 0$ in $X'$ as $k\to \infty$ has a convergent subsequence.
\end{definition}

\begin{theorem}[{\cite[Theorem 9.12]{Rabinowitz86}}]\label{Rabinowitz86}
Let $f$ be a even $C^1$ function, defined on an infinite dimensional Banach space $X$ and such that $f(0) = 0$. Assume that $X$ is decomposable as direct sum of two closed subspaces $X = X_1 \oplus X_2$, with $\dim X_1 < \infty$.
Suppose that:
\begin{itemize}
\item [$(i)$] there exist $\delta,\varrho>0$ such that
\begin{equation}
\inf f(S_\varrho\cap X_2)\geq\delta,
\end{equation}
where $S_{\varrho}:=\{u\in X\,:\,\|u\|_X=\varrho\}$;
    \item[$(ii)$] for any finite dimensional subspace $Y\subset X$ there exists $R=R(Y)>0$ such that for any $u\in Y$ with $\|u\|\geq R$
    \begin{equation}
    f(u)\leq 0\,;
    \end{equation}
    \item[$(iii)$] $f$ satisfies the $(PS)$ condition.
  \end{itemize}
  Then, $f$ has an unbounded sequence of positive critical values.
\end{theorem}

Notice that for the functional $J\colon W\to\R$ defined in~\eqref{eq:J} we have
\begin{itemize}
  \item $J\in C^1(W)$;
  \item $J(0)=0$;
  \item $J$ is even.
\end{itemize}
In the next two sections we prove that $J$, or a suitable restriction of it, satisfies the assumptions of Theorem~\ref{Rabinowitz86} in both the cases $(\text{I})$ and $(\text{II})$.
The appropriate choice of the functional will differ in the two cases.


\section{Case (\text{I}): the generalized KGM equation}
\label{Sect3}
In this section we prove our existence result when
\begin{equation}
V(x)=m^2\quad\text{for all $x\in\R^3$},
\end{equation}
with $m>0$. In this case $W=H^1(\R^3)$ and
$$\|u\|_W=\|u\|_{H^1}\quad\text{for all $u\in W=H^1(\R^3)$}.$$

To find critical points of $J$ we shall restrict it to the subspace of radial functions
\[
H_r^1(\R^3):=\{u\in H^1(\R^3)\,:\,u(|x|)=u(x)\text{ for any }x\in\R^3\}.
\]
This (standard) procedure is allowed by the following result.
\begin{lemma}\label{Lemmaradial}
Under the assumptions of Theorem~\ref{mainthm1}, $u\in H^1_r(\mathbb{R}^3)$ is a critical point of $J|_{H^1_r(\mathbb{R}^3)}$ if and only if $u$ is a critical point of $J$.
\end{lemma}
\begin{proof}
The arguments of~\cite[Lemma 4.2]{BF} continue to apply to the functional $J$ defined in~\eqref{eq:J}, which is given by
\begin{equation}
J(u):=\dfrac{1}{2}\mathcal B_\alpha(u,u) +\dfrac{m^2 -\omega^2}{2}\int_{\R^3}u^2\,dx+\dfrac{\omega}{2}\int_{\R^3}\varphi_u u^2\,dx-\frac{1}{p}\int_{\R^3}|u|^p \,dx.
\end{equation}
Considering that the function $\varphi_u$ in the paper~\cite{BF} corresponds to $-\varphi_u$ in the present one, $J$ can be written as
\begin{equation}
J(u):=\dfrac{\alpha}{2}\|(-\Delta)^\frac{s}{2}u\|_2^2+J_0(u),
\end{equation}
where $J_0$ is the functional in~\cite{BF}.

The main argument there consists in the invariance of $J_0$ under to $O(3)$ group action $T_g$ on $H^1(\R^3)$ given by
\begin{equation}
    T_g u(x)=u(g(x)),\quad g\in O(3),
\end{equation}
explicitly written as $g(x)=Ox$, where $O$ is an orthogonal matrix.

Now also $u\to\|(-\Delta)^\frac{s}{2}u\|_2^2$ is invariant under the same action.
This fact can be established by using the Spectral Theorem or, in a more elementary way, by~\cite[Proposition 3.6]{DRV}, that is the formula
\begin{equation}
\|(-\Delta)^{\frac{s}{2}}u\|_2^2=\int_{\R^3}|\xi|^{2s}|\mathcal Fu(\xi)|^2\,d\xi.
\end{equation}
Indeed, for any $u\in\mathcal S(\R^3)$ one has
\begin{equation}
\mathcal F(T_g u)=T_g(\mathcal Fu),
\end{equation}
and then
\begin{equation}
\begin{split}
    \|(-\Delta)^{\frac{s}{2}}T_g u\|_2^2&=\int_{\R^3}|\xi|^{2s}|\mathcal F (T_g u)(\xi)|^2\,d\xi=\int_{\R^3}|\xi|^{2s}|T_g(\mathcal F u(\xi))|^2\,d\xi\\
    &=\int_{\R^3}|\xi|^{2s}|\mathcal F u(O\xi)|^2\,d\xi=\int_{\R^3}|\xi|^{2s}|\mathcal Fu(\xi)|^2\,d\xi=\|(-\Delta)^{\frac{s}{2}}u\|_2^2\,,
\end{split}
\end{equation}
since $O$ is orthogonal.
\end{proof}
We shall then use Theorem~\ref{Rabinowitz86}, with $X=H^1_r(\R^3)$, $X_1=\{0\}$, and $X_2=X$.
\begin{lemma}\label{condgeome1}
Under the assumptions of Theorem~\ref{mainthm1} the functional $J$ satisfies $(i)$ and $(ii)$ of Theorem~\ref{Rabinowitz86} in $X=H^1(\R^3)$, with $X_1=\{0\}$ and $X_2=H^1(\R^3)$, and consequently also in $X=H^1_r(\R^3)$, with $X_1=\{0\}$ and $X_2=H^1_r(\R^3)$.
\end{lemma}
\begin{proof}
We first claim that there exist $\delta,\varrho>0$ such that
\begin{equation}\label{stimabasso}
\inf J(S_\varrho) \ge \delta,
\end{equation}
where $S_{\varrho}:=\{u\in H^1(\R^3) \,:\, \|u\|_{H^1}=\varrho\}$.
Indeed, by~\eqref{young} and by Lemma \ref{lem:varphi}, for any $u\in H^1(\R^3)$ and $\e>0$ we have
\begin{align*}
J(u)&=\mathcal B_\alpha(u,u)+ \frac{m^2-\omega^2}{2}\|u\|^2_2+\frac{\omega}{2}\int_{\R^3}\varphi_u u^2\,dx-\frac{1}{p}\|u\|^p_p\\
&\ge \frac{1}{2}\|\nabla u\|^2_2-\frac{\alpha^{-}}{2}\left(s\e \|\nabla u\|^2_2+(1-s)\e^{-\frac{s}{1-s}}\|u\|^2_2\right)+ \frac{m^2-\omega^2}{2}\|u\|^2_2-\frac{1}{p}\|u\|^p_p\\
    &\ge \frac{1}{2}\left(1-\alpha^{-}s\e\right)\|\nabla u\|^2_2+\frac{1}{2}\left(m^2-\omega^2 - \alpha^{-} (1-s)\e^{-\frac{s}{1-s}}\right)\|u\|^2_2-\frac{1}{p}\|u\|^p_p\,,
\end{align*}
where $\alpha^-:=\max\{-\alpha,0\}$ denotes the negative part of $\alpha$.
Let us consider the following system
\begin{equation}\label{eq:system}
\begin{cases}
1-\alpha^-s\e>0,\\
m^2-\omega^2-\alpha^-(1-s)\e^{-\frac{s}{1-s}}>0.
\end{cases}
\end{equation}
The first inequality of~\eqref{eq:system} holds when $\alpha^-=0$ and, elsewhere, trivially gives
\begin{equation}
  \e<\frac{1}{\alpha^-s}.
\end{equation}
The second inequality of~\eqref{eq:system} leads us to
\begin{align*}
\alpha^-(1-s)\e^{-\frac{s}{1-s}}&<m^2-\omega^2.
\end{align*}
Since by assumption
$$m^2-\omega^2\ge m^2-\omega^2-\frac{(4-p)^+}{p-2}w^2=:\Omega>0,$$
the system~\eqref{eq:system} is satisfied whenever
$$\frac{(1-s)^{\frac{1-s}{s}}\left(\alpha^-\right)^{\frac{1-s}{s}}}{(m^2-\omega^2)^{\frac{1-s}{s}}}<\e<\frac{1}{\alpha^-s}.$$
Thanks to $\alpha>-\alpha_0(s,\Omega)$, which implies
\begin{equation*}
\alpha^{-}<\alpha_0(s,\Omega)=s^{-s}(1-s)^{s-1}\Omega^{1-s}\le s^{-s}(1-s)^{s-1}(m^2-\omega^2)^{1-s},
\end{equation*}
there exists $\e_0\in(0,\infty)$ such that
\begin{equation}\label{eq:c1c2}
  c_1:=1-\alpha^{-}s\e_0>0,\quad c_2:=m^2-\omega^2 - \alpha^{-} (1-s)\e_0^{-\frac{s}{1-s}}>0.
\end{equation}
Hence we get
\begin{align}\label{eq:below}
    \dfrac{1}{2}\mathcal B_\alpha(u,u) +\dfrac{m^2 -\omega^2}{2}\int_{\R^3}u^2\,dx\ge& \frac{1}{2}\min\{c_1,c_2\}\|u\|^2_{H^1}.
\end{align}
Therefore, by using also~\eqref{eq:Sobolev}, for any $u\in S_{\varrho}$, we have
\begin{align*}
    J(u)\ge& \frac{1}{2}\min\{c_1,c_2\}\|u\|^2_{H^1}-\frac{C_p^p}{p}\|u\|^p_{H^1}\\
    =& \frac{1}{2}\min \{c_1,c_2\} \cdot \varrho^2-\frac{C_p^p}{p} \cdot \varrho^p\\
    =&\varrho^2 \left( \frac{1}{2}\min \{c_1,c_2\}-\frac{C_p^p}{p}\cdot \varrho^{p-2}\right)>0,
\end{align*}
where the last inequality is given by eventually setting
\begin{equation*}
\varrho<\left(\frac{p\min \{c_1,c_2\}}{2C_p^p}\right)^\frac{1}{p-2}.
\end{equation*}
Thus, $(i)$ is satisfied.

Let us prove $(ii)$. We fix a finite dimensional space $Y\subset H^1(\R^3)$ and $u\in Y$. By~\eqref{young} there exists a positive constant $K>0$ such that
\begin{equation}\label{stimaalto}
J(u)\le K\|u\|_{H^1}^2-\frac{1}{p}\|u\|_p^p\to-\infty
\end{equation}
as $\|u\|_{H^1}\to \infty$, since all the norms are equivalent.

Trivially we can replace $H^1(\R^3)$ with $H^1_r(\R^3)$ in~\eqref{stimabasso} and~\eqref{stimaalto}, completing the proof.
\end{proof}
\begin{lemma}\label{PS1}
Under the assumptions of Theorem~\ref{mainthm1} the functional $J|_{H^1_r(\R^3)}$ satisfies $(iii)$ of Theorem~\ref{Rabinowitz86}.
\end{lemma}
\begin{proof}
Let us fix a $(PS)$ sequence $(u_n)_n\subset H^1_r(\R^3)$. Then, for all $p\in(2,6)$, we have
\begin{equation}
\begin{split}
  pJ(u_n)-J'(u_n)[u_n]&\ge\left(\frac{p}{2}-1\right)\left(\mathcal B_\alpha(u_n,u_n)+(m^2-\omega^2)\|u_n\|_2^2\right)\\
&\quad+\omega\left(\frac{p}{2}-2\right)\int_{\R^3}\varphi_{u_n}u_n^2\,dx.
\end{split}
\end{equation}
Note that the presence of the last term on the r.h.s. force us to distinguish between two possible cases: the case $2<p<4$ and the case $4\le p<6$.

\noindent {\bf Case 1}.
In the case $4\le p<6$ and $m>\omega>0$, by~\eqref{eq:below} we immediately get
\begin{equation}\label{case1lower}
  pJ(u_n)-J'(u_n)[u_n]\ge \left(\frac{p}{2}-1\right)\min\{c_1,c_2\}\|u_n\|_{H^1}^2,
\end{equation}
being $c_1>0$ and $c_2>0$ the two constants defined in~\eqref{eq:c1c2}.

\noindent {\bf Case 2}.
Assume now $2<p<4$ and $m\sqrt{p-2}> \omega \sqrt{2}(>0)$. Then, being $p/2-2<0$ and $-\varphi_{u_n}\ge-\omega$, for every $\e>0$ we get
\begin{equation*}
\begin{aligned}
  pJ(u_n)-J'(u_n)[u_n]&\ge \left(\frac{p}{2}-1\right)\mathcal B_{\alpha}(u_n,u_n)\\&\quad+\left[\left(\frac{p}{2}-1\right)m^2-\left(\frac{p}{2}-1\right)\omega^2+\left(\frac{p}{2}-2\right)\omega^2\right]\|u_n\|_2^2\\
  &= \left( \frac{p}{2}-1\right)\mathcal B_\alpha (u_n,u_n)+\left[\left(\frac{p}{2}-1\right)m^2-\omega^2\right]\|u_n\|_2^2\\
  &\geq\left(\frac{p}{2}-1\right)\left(1-\alpha^-s\e\right)\|\nabla u_n\|_2^2\\
  &\quad+\left[\left(\frac{p}{2}-1\right)\left(m^2-\alpha^-(1-s)\e^{-\frac{s}{1-s}}\right)-\omega^2\right]\|u_n\|_2^2.
\end{aligned}
\end{equation*}
We consider now the following system
\begin{equation}\label{system}
  \begin{cases}
    1-\alpha^-s\e>0,\\
    \left(\frac{p}{2}-1\right)\left(m^2-\alpha^-(1-s)\e^{-\frac{s}{1-s}}\right)>\omega^2.
  \end{cases}
\end{equation}
As before, the first inequality of~\eqref{system} holds when either $\alpha^-=0$ or
\begin{equation}
  \e<\frac{1}{\alpha^-s}.
\end{equation}
The second inequality of~\eqref{system} leads us to
\begin{align*}
  \alpha^-(1-s)\e^{-\frac{s}{1-s}}&<m^2-\frac{2\omega^2}{p-2}=
  m^2-\omega^2-\left(\frac{4-p}{p-2}\right)\omega^2=\Omega.
\end{align*}
Since $\Omega>0$ the system~\eqref{system} is satisfied whenever
$$\frac{(1-s)^{\frac{1-s}{s}}\left(\alpha^-\right)^{\frac{1-s}{s}}}{\Omega^{\frac{1-s}{s}}}<\e<\frac{1}{\alpha^-s}.$$
By noticing that $\alpha^-<\alpha_0(s,\Omega)$, we can find $\e_1\in(0,\infty)$ such that
\begin{equation}
  d_1:=\left(\frac{p}{2}-1\right)\left(1-\alpha^-s\e_1\right),\quad d_2:=\left(\frac{p}{2}-1\right)\left(m^2-\alpha^-(1-s)\e_1^{-\frac{s}{1-s}}\right)-\omega^2
\end{equation}
are both positive constants. Therefore
\begin{equation}\label{boundedness}
  pJ(u_n)-J'(u_n)[u_n]\ge\min\{d_1,d_2\}\|u_n\|_{H^1}^2.
\end{equation}

Since $(J(u_n))_n$ is bounded in $\R$ and $(J'|_{H_r^1(\R^3)}(u_n))_n$ is bounded in $(H^1_r(\R^3))'$, being $(u_n)_n$ a $(PS)$ sequence, there exist two positive constants $K_1,K_2$ such that
\begin{equation}\label{PSseq}
J(u_n)\le\ K_1\quad\text{and}\quad|J'(u_n)[u_n]|\le K_2\|u_n\|_{H^1}\quad\text{for any }n\in\mathbb{N}.
\end{equation}
Hence, by~\eqref{case1lower} and~\eqref{boundedness}, setting
\begin{align*}
    c_3=\begin{cases}
    \left(\frac{p}{2}-1\right)\min\{c_1,c_2\},&\quad\text{in Case }1,\\
    \min\{d_1,d_2\},&\quad\text{in Case }2,
  \end{cases}
\end{align*}
we get
\begin{equation*}
  pK_1+K_2\|u_n\|_{H^1}\ge c_3\|u_n\|_{H^1}^2\quad\text{for any }n\in\mathbb{N},
\end{equation*}
which implies that $(u_n)_n$ is bounded in $H^1_r(\R^3)$.



Therefore, there exist a subsequence, not relabeled, and $u\in H_r^1(\R^3)$ such that $(u_n)_n$ converges to $u$ weakly in $H_r^1(\R^3)$, strongly in $L^p(\R^3)$ for any $p\in(2,6)$, and a.e. in $\R^3$. To conclude we show that the convergence in $H_r^1(\R^3)$ turns out to be strong.

By \eqref{eq:idvarphiu}, we deduce that the sequence $(\varphi_{u_n})_n\subset\mathcal D^{1,2}(\R^3)$ satisfies for all $n\in\N$
\begin{equation*}
\|\nabla \varphi_{u_n}\|_2^2\le \omega\int_{\R^3}\varphi_{u_n}u_n^2\,dx\le \omega\|\varphi_{u_n}\|_6\|u_n\|_\frac{12}{5}^2\le c_5\|\nabla\varphi_{u_n}\|_2\|u_n\|_{H^1}^2,
\end{equation*}
which implies that $(\varphi_{u_n})_n$ is bounded in $\mathcal D^{1,2}(\R^3)$. Moreover, by~\eqref{eq:below} we have
\begin{align*}
\min\{c_1,c_2\}\|u_n-u\|_{H^1}^2&\le\mathcal B_{\alpha}(u_n-u,u_n-u)+(m^2-\omega^2)\|u_n-u\|_2^2\\
&=J'(u_n)[u_n-u]-J'(u)[u_n-u]\\
&\quad-2\omega\int_{\R^3}(\varphi_{u_n}u_n-\varphi_u u)(u_n-u)\,dx\\
&\quad +\int_{\R^3}(\varphi_{u_n}^2u_n-\varphi_u^2 u)(u_n-u)\,dx\\
&\quad+\int_{\R^3}(|u_n|^{p-2}u_n-|u|^{p-2}u)(u_n-u)\,dx,
\end{align*}
being $c_1>0$ and $c_2>0$ the two constants defined in~\eqref{eq:c1c2}.
Since $J'(u_n)\to 0$ in $(H^1_r(\R^3))'$ and $u_n\rightharpoonup u$ in $H^1_r(\R^3)$ as $n\to\infty$, it follows that the first two terms converge to 0 as $n\to\infty$. Moreover, as $n\to \infty$
\begin{align*}
&\left|\int_{\R^3}(\varphi_{u_n}u_n-\varphi_u u)(u_n-u)\,dx\right|\le
\left(\|\varphi_{u_n}\|_6\|u_n\|_\frac{12}{5}+\|\varphi_u\|_6\|u\|_\frac{12}{5}\right)\|u_n-u\|_\frac{12}{5}\to 0,\\
&\left|\int_{\R^3}(\varphi_{u_n}^2u_n-\varphi_u^2 u)(u_n-u)\,dx\right|\le
\left(\|\varphi_{u_n}\|_6^2\|u_n\|_3+\|\varphi_u\|_6^2\|u\|_3\right)\|u_n-u\|_3\to 0,\\
&\left|\int_{\R^3}(|u_n|^{p-2}u_n-|u|^{p-2}u)(u_n-u)\,dx\right|\le(\|u_n\|_p^{p-1}+\|u_n\|_p^{p-1})\|u_n-u\|_p\to 0.
\end{align*}
Hence the thesis follows.
\end{proof}
We can now conclude this section by collecting all the results given in the previous lines in the proof of Theorem~\ref{mainthm1}.
\begin{proof}[Proof of Theorem~\ref{mainthm1}]
    The proof follows by Lemmas~\ref{Lemmaradial}--\ref{PS1} and Theorem~\ref{Rabinowitz86}.
\end{proof}


\section{Case (\text{II}): the KGM equation with external potential}
\label{Sect4}

In this section we consider problem~\eqref{eq:problem} in the case $(\text{II})$, that is when the potential $V\in C(\R^3)$, with $V_0:=\inf_{x\in\R^3}V(x)>-\infty$, satisfies~\eqref{V}.
Under these assumptions, by the same  arguments used in~\cite{BW} (see also~\cite{PXZ}), we can prove the following compactness result for the space $W$.

\begin{lemma}\label{lem:com}
Assume that $V\in C(\R^3)$, with $V_0:=\inf_{x\in\R^3}V(x)>-\infty$, satisfies~\eqref{V}. Then for all $p\in[2,6)$ the embedding $W\subset L^p(\R^3)$ is compact.
\end{lemma}

\begin{proof}
We first consider the case $p=2$. Let $(u_k)_k\subset W$ be a bounded sequence in $W$. Then there exists a subsequence $(u_{k_j})_j$ and a function $u\in W$ such that $u_j:=u_{k_j}\rightharpoonup u$ weakly in $W$ as $j\to\infty$. Moreover, there exists a positive constant $C$ such that
\begin{equation*}
\|u_j\|_W+\|u\|_W\le C\quad\text{for all $j\in\N$}.
\end{equation*}
For all fixed $R>0$ we have that $u_j\rightarrow u$ in $L^2(B_R(0))$ as $j\to\infty$, since $W\subset H^1(\R^3)$ and the embedding $H^1(\R^3)\subset L^2(B_R(0))$ is compact.
Hence, it remains to estimate the integral
\begin{equation*}
\int_{\R^3\setminus B_R(0)}|u_j-u|^2\,dx.
\end{equation*}

For all fixed $M>V_0$, we set
\begin{equation*}
A_1(y):=\{x\in B_h(y)\,:\,V(x)\le M\},\quad A_2(y):=\{x\in B_h(y)\,:\,V(x)>M\},
\end{equation*}
where $h>0$ is the constant independent of $M$ given by~\eqref{V}. We choose a sequence of points $(y_i)_i\subset\R^3$ such that $\R^3=\cup_{i=1}^\infty B_h(y_i)$ and each $x\in\R^3$ is covered by at most $2^3=8$ of such balls. We have
\begin{align*}
\int_{\R^3\setminus B_R(0)}|u_j-u|^2\,dx&\le\sum_{|y_i|\ge R-h}\int_{B_h(y_i)}|u_j-u|^2\,dx\\
&=\sum_{|y_i|\ge R-h}\left(\int_{A_1(y_i)}|u_j-u|^2\,dx+\int_{A_2(y_i)}|u_j-u|^2\,dx\right).
\end{align*}
We separately estimate these two integrals. For the second one we have
\begin{equation*}
\int_{A_2(y_i)}|u_j-u|^2\,dx\le\frac{1}{M-V_0}\int_{B_h(y_i)}(V-V_0)|u_j-u|^2\,dx.
\end{equation*}
To estimate the first one we fix $q\in(2,6)$.
By H\"older's inequality we then get
\begin{align*}
\int_{A_1(y_i)}|u_j-u|^2\,dx&\le\|1\|_{L^\frac{q}{q-2}(A_1(y_i))}\||u_j-u|^2\|_{L^{\frac{q}{2}}(A_1(y_i))}\\
&\le |A_1(y_i)|^{\frac{q-2}{q}}\|u_j-u\|^2_{L^q(B_h(y_i))}.
\end{align*}
Therefore
\begin{align*}
&\int_{\R^3\setminus B_R(0)}|u_j-u|^2\,dx\\
&\le\!\!\sum_{|y_i|\ge R-h}\!\!\left(\frac{1}{M-V_0}\int_{B_h(y_i)}(V-V_0)|u_j-u|^2\,dx+|A_1(y_i)|^{\frac{q-2}{2}}\|u_j-u\|^2_{L^q(B_h(y_i))}\right)\\
&\le\frac{8}{M-V_0}\int_{\R^3} (V-V_0)|u_j-u|^2\,dx+\!\!\!\!\sup_{|y_i|\ge R-h}\!\!\!\!|A_1(y_i))|^{\frac{q-2}{q}}\!\!\!\!\sum_{|y_i|\ge R-h}\!\!\!\!\|u_j-u\|^2_{L^q(B_h(y_i))}.
\end{align*}
Since $W\subset H^1(B_h(y))$ and the embedding $H^1(B_h(y))\subset L^q(B_h(y))$ is continuous, for all $h>0$ and $y\in\R^3$ we can find a constant $C_h=C_h(q)>0$ such that
\begin{equation*}
\|u\|_{L^q(B_h(y))}\le C_h\|u\|_{H^1(B_h(y))}\quad\text{for all $u\in W$}.
\end{equation*}
Notice that $C_h$ is independent of $y\in\R^3$. Hence, we can estimate
\begin{align*}
\sum_{|y_i|\ge R-h}\|u_j-u\|^2_{L^q(B_h(y_i))}&\le C_h^2\sum_{|y_i|\ge R-h}\|u_j-u\|_{H^1(B_h(y_i))}^2\\
&\le 8C_h^2\|u_j-u\|_{H^1(\R^3)}^2\le 8C_h^2\|u_j-u\|^2_W.
\end{align*}
By combining the above inequalities we get
\begin{align*}
&\int_{\R^3\setminus B_R(0)}|u_j-u|^2\,dx\\
&\le\frac{8}{M-V_0}\int_{\R^3}(V-V_0)|u_j-u|^2\,dx+8C_h^2\sup_{|y_i|\ge R-h}|A_1(y_i)|^{\frac{q-p}{q}}\|u_j-u\|^2_W\\
&\le\frac{8}{M-V_0}\left(\|u_j\|_W+\|u\|_W\right)^2+8C_h^2\sup_{|y|\ge R-h}|A_1(y)|^{\frac{q-2}{q}}\left(\|u_j\|_W+\|u\|_W\right)^2\\
&\le\frac{8C^2}{M-V_0}+8C_h^2C^2\sup_{|y|\ge R-h}|A_1(y)|^{\frac{q-2}{q}}.
\end{align*}
Let us fix $\varepsilon>0$ and choose $M$ so large that
$$\frac{8C^2}{M-V_0}<\frac{\varepsilon}{3}.$$
For such a fixed $M$, by $\eqref{V}$ there exists $R_M>0$ such that
\begin{equation*}
8C_h^2 C^2\sup_{|y|\ge R_M-h}|A_1(y)|^{\frac{q-2}{q}}<\frac{\varepsilon}{3}.
\end{equation*}
Furthermore, since $u_j\rightarrow u$ strongly in $L^2(B_{R_M}(0))$, there exists $j_0\in\N$ such that
\begin{equation*}
\int_{B_{R_M}(0)}|u_j-u|^2\,dx<\frac{\varepsilon}{3}, \quad\text{for all $j\ge j_0$}.
\end{equation*}
Thus,
\begin{equation*}
\lim_{j\rightarrow\infty}\int_{\R^3}|u_j-u|^2\,dx=0,
\end{equation*}
and so $u_j\rightarrow u$ strongly in $L^2(\R^3)$.

For all $p\in(2,6)$ we take $\theta\in (0,1)$ such that
\begin{equation*}
\frac{1}{p}=\frac{\theta}{2}+\frac{1-\theta}{6}.
\end{equation*}
Then, by using the H\"older's inequality we have
\begin{equation*}
\|u_j-u\|_{L^p(\R^3)}\le\|u_j-u\|^\theta_{L^2(\R^3)}\|u_j-u\|^{1-\theta}_{L^6(\R^3)}\rightarrow 0
\end{equation*}
as $j\rightarrow\infty$, since $u_j\rightarrow u$ in $L^2(\R^3)$ and $(u_j)_j$ is bounded in $W\subset L^6(\R^3)$.
\end{proof}

In order to prove that the functional $J$ satisfies the geometric assumptions $(i)$ and $(ii)$ of the Theorem~\ref{Rabinowitz86}, we introduce a new operator $\mathcal L_{\alpha,V}\colon W\to W'$, defined as
$$\mathcal L_{\alpha,V}u=\mathcal L_\alpha u+Vu\quad\text{for $u\in W$}.$$
As done in Definition~\ref{def:Balpha}, we can naturally associate to $\mathcal L_{\alpha,V}$ a bilinear form $\mathcal B_{\alpha,V}\colon W\times W\to \R$ as
\begin{align*}
\mathcal B_{\alpha,V}(u,v):=\mathcal B_\alpha(u,v)+\int_{\R^3}Vuv\,dx\quad\text{for all $u,v\in W$}.
\end{align*}
Notice that $\mathcal B_{\alpha,V}$ is continuous on $W\times W$ and by~\eqref{young} for all $\alpha\in\R$ there exists a constant $\gamma =\gamma (s,\alpha,V_0)\ge 0$ such that
\begin{equation}\label{eq:below2}
\mathcal B_{\alpha,V}(u,u)+\gamma \|u\|_2^2\ge \frac{1}{2}\|u\|_W^2\quad\text{for all $u\in W$}.
\end{equation}
Since the embedding $W\subset L^2(\R^3)$ is continuous, dense, and compact, thanks to Lemma~\ref{lem:den} and Lemma~\ref{lem:com}, we can apply the spectral decomposition result given in Proposition~\ref{thm:spc}.  Hence, there exists an increasing sequence $(\lambda_k)_k$ of eigenvalues of $\mathcal L_{\alpha,V}$ satisfying
\begin{equation}
-\gamma < \lambda_1\le \lambda_2\le\cdots\le \lambda_k\to \infty\quad\text{as $k\to\infty$}.
\end{equation}
Moreover, for all $k\in \mathbb N$ the eigenvalue $\lambda_k$ has finite multiplicity and there exists a sequence of eigenvectors $(e_k)_k\subset W$ corresponding to $(\lambda_k)_k$, which is an orthonormal basis of $L^2(\R^3)$.
In particular, as show in Remark~\ref{rem:dec}, if we define
\begin{equation}
H_1:=\{0\}, \quad\mathbb P_1:=W,
\end{equation}
and for all $k\ge 2$
\begin{align}
H_k:=&\textrm{span}\{e_1,\dots,e_{k-1}\}\subset W,\\
\mathbb P_k:=&\left\{u\in W\,:\, \int_{\R^3}ue_j=0\text{ for all $j=1,\dots,k-1$}\right\},
\end{align}
then $W$ is decomposable as direct sum of these two closed subspace $W=H_k\oplus \mathbb P_k$ for all $k\in\N$, with $\dim H_k=k-1<\infty$.

Let $k_0\in\N$ be such that
\begin{equation}\label{k0}
\lambda_{k_0}>\omega^2.
\end{equation}
Then, there exists a constant $c_0=c_0(s,\alpha,\omega,V_0)>0$ satifying
\begin{align}\label{eq:ck0}
\mathcal B_{\alpha,V}(u,u)-\omega^2\|u\|_2^2\ge c_0\|u\|_W^2\quad\text{for all $u\in \mathbb P_{k_0}$}.
\end{align}
Indeed, in view of~\eqref{eq:below2} and \eqref{eq:lambdak}, for all $u\in\mathbb P_{k_0}$ we have
\begin{align*}
\mathcal B_{\alpha,V}(u,u)-\omega^2\|u\|_2^2&=\mathcal B_{\alpha,V}(u,u)+\gamma \|u\|_2^2-(\omega^2+\gamma)\|u\|_2^2\\
&=\left(1-\frac{\omega^2+\gamma }{\lambda_{k_0}+\gamma }\right)\left(\mathcal B_{\alpha,V}(u,u)+\gamma \|u\|_2^2\right)\\
&\quad+\left(\frac{\omega^2+\gamma}{\lambda_{k_0}+\gamma }\right)\left(\mathcal B_{\alpha,V}(u,u)+\gamma \|u\|_2^2\right)-(\omega^2+\gamma)\|u\|_2^2\\
&\ge\left(1-\frac{\omega^2+\gamma }{\lambda_{k_0}+\gamma }\right)\|u\|_W^2=:c_0\|u\|_W^2.
\end{align*}
\begin{lemma}\label{condgeome2}
Under the assumptions of Theorem~\ref{mainthm2} the functional $J$ satisfies $(i)$ and $(ii)$ of Theorem~\ref{Rabinowitz86} in $X=W$, with $X_1=H_{k_0}$ and $X_2=\mathbb P_{k_0}$, where $\lambda_{k_0}$ is given by \eqref{k0}.
\end{lemma}
\begin{proof}
Let us prove $(i)$. By~\eqref{eq:Sobolev} and~\eqref{eq:ck0}, for all $u\in X_2=\mathbb P_{k_0}$ we have
\begin{equation}
J(u)\ge \frac{1}{2}\mathcal B_{\alpha,V}(u,u)-\frac{\omega^2}{2}\|u\|_2^2-\frac{1}{p}\|u\|_p^p\ge\left(\frac{c_0}{2}-\frac{C_p^p}{p}\|u\|_W^{p-2}\right)\|u\|_W^2.
\end{equation}
Hence, arguing as in Lemma~\ref{condgeome1}, there exist $\delta,\varrho>0$ such that
\begin{equation}
\inf J(S_\varrho\cap X_2)\ge\delta,
\end{equation}
where $S_\varrho:=\{u\in W\,:\, \|u\|_W=\varrho\}$.

Let us prove $(ii)$. By~\eqref{young} there exists a constant $K>0$ such that for all finite dimensional space $Y\subset W$ and $u\in Y$ we have
\begin{equation}
J(u)\le K\|u\|_W^2-\frac{1}{p}\|u\|_p^p\to-\infty
\end{equation}
as $\|u\|_W\to \infty$, since all the norms are equivalent.
\end{proof}

\begin{lemma}\label{PS2}
Under the assumptions of Theorem~\ref{mainthm2} the functional $J$ satisfies $(iii)$ of Theorem~\ref{Rabinowitz86}.
\end{lemma}
\begin{proof}
Let $(u_n)_n\subset W$ be a $(PS)$ sequence. By using Lemma~\ref{lem:varphi}, the lower bound~\eqref{eq:below2},  and that $p\in(2,6)$, for all $n\in\mathbb N$ we have
\begin{align*}
&pJ(u_n)-J'(u_n)[u_n]\\
&=\left(\frac{p}{2}-1\right) \mathcal B_{\alpha,V}(u_n,u_n)-\omega^2\left(\frac{p}{2}-1\right)\int_{\R^3}u_n^2\,dx\\
&\quad+\omega\left(\frac{p}{2}-2\right)\int_{\R^3}\varphi_{u_n}u_n^2\,dx+\int_{\R^3}\varphi_{u_n}^2 u_n^2\,dx\\
&\ge\frac{1}{2}\left(\frac{p}{2}-1\right)\|u_n\|_W^2-(\gamma+\omega^2)\left(\frac{p}{2}-1\right)\|u_n\|_2^2-\omega\int_{\R^3}\varphi_{u_n}u_n^2\,dx\\
&\ge c_1\|u_n\|_W^2-c_2\|u_n\|_2^2,
\end{align*}
for two constants $c_1,c_2>0$.

Assume by contradiction that $\|u_n\|_W\to \infty$ and define $w_n:=\frac{u_n}{\|u_n\|_W}$ for all $n\in\N$. Since $\|w_n\|_W=1$ for all $n\in\N$, by Lemma~\ref{lem:com} there exist a subsequence, not relabeled, and a function $w\in W$ such that as $n\to\infty$
\begin{equation}
w_n\rightharpoonup w\text{ in $W$},\quad w_n\to w\text{ in $L^p(\R^3)$ for all $p\in [2,6)$}.
\end{equation}
In particular, since
\begin{equation}
c_1-c_2\|w_n\|_2^2\le \frac{pJ(u_n)}{\|u_n\|_W^2}-\frac{J'(u_n)[u_n]}{\|u_n\|_W^2}\to 0\quad\text{as $n\to\infty$}
\end{equation}
we have
\begin{equation}
\|w\|_2^2\ge \frac{c_1}{c_2}>0.
\end{equation}
On the other hand, by Lemma~\ref{lem:varphi}, for all $n\in\N$ we have
\begin{equation}
\frac{1}{p}\|u_n\|_p^p\le \frac{1}{2}\mathcal B_{\alpha,V}(u_n,u_n)-J(u_n).
\end{equation}
Therefore, since $(u_n)_n$ is a $(PS)$ sequence, we deduce
\begin{equation}
0<\frac{1}{p}\|w_n\|_p^p\le \frac{1}{2}\frac{\mathcal B_{\alpha,V}(u_n,u_n)}{\|u_n\|_W^p}+\frac{|J(u_n)|}{\|u_n\|_W^p}\le \frac{c_3}{\|u_n\|_W^{p-2}}+\frac{c_4}{\|u_n\|_W^p}\to 0
\end{equation}
as $n\to\infty$, with $c_3,c_4>0$. Hence $w=0$, which leads to a contradiction.

Since the sequence $(u_n)_n\subset W$ is bounded, there exist a subsequence, not relabeled, and $u\in W$ such that as $n\to\infty$
\begin{equation}
u_n\rightharpoonup u\text{ in $W$},\quad u_n\to u\text{ in $L^p(\R^3)$ for all $p\in [2,6)$}.
\end{equation}
By~\eqref{eq:idvarphiu} also the sequence $(\varphi_{u_n})_n\subset\mathcal D^{1,2}(\R^3)$ is bounded, since for all $n\in\mathbb N$
\begin{equation}
\|\nabla \varphi_{u_n}\|_2^2\le \omega\int_{\R^3}\varphi_{u_n}u_n^2\,dx\le \omega\|\varphi_{u_n}\|_6\|u_n\|_\frac{12}{5}^2\le c_5\|\nabla\varphi_{u_n}\|_2\|u_n\|_W^2.
\end{equation}
We claim that $u_n\to u$ in $W$ as $n\to\infty$. Indeed, by~\eqref{eq:below2} for all $n\in\mathbb N$ we have
\begin{align*}
\frac{1}{2}\|u_n-u\|_W^2&\le\mathcal B_{\alpha,V}(u_n-u,u_n-u)+\gamma\|u_n-u\|_2^2\\
&=J'(u_n)[u_n-u]-J'(u)[u_n-u]+(\gamma+\omega^2) \|u_n-u\|_2^2\\
&\quad -2\omega\int_{\R^3}(\varphi_{u_n}u_n-\varphi_u u)(u_n-u)\,dx\\
&\quad +\int_{\R^3}(\varphi_{u_n}^2u_n-\varphi_u^2 u)(u_n-u)\,dx\\
&\quad+\int_{\R^3}(|u_n|^{p-2}u_n-|u|^{p-2}u)(u_n-u)\,dx.
\end{align*}
Since $J'(u_n)\to 0$ in $W'$, $u_n\rightharpoonup u$ in $W$, and $u_n\to u$ in $L^2(\R^3)$ as $n\to\infty$, it follows that the first three terms converge to 0 as $n\to\infty$. Moreover, as $n\to \infty$
\begin{align*}
&\left|\int_{\R^3}(\varphi_{u_n}u_n-\varphi_u u)(u_n-u)\,dx\right|\le
\left(\|\varphi_{u_n}\|_6\|u_n\|_\frac{12}{5}+\|\varphi_u\|_6\|u\|_\frac{12}{5}\right)\|u_n-u\|_\frac{12}{5}\to 0,\\
&\left|\int_{\R^3}(\varphi_{u_n}^2u_n-\varphi_u^2 u)(u_n-u)\,dx\right|\le
\left(\|\varphi_{u_n}\|_6^2\|u_n\|_3+\|\varphi_u\|_6^2\|u\|_3\right)\|u_n-u\|_3\to 0,\\
&\left|\int_{\R^3}(|u_n|^{p-2}u_n-|u|^{p-2}u)(u_n-u)\,dx\right|\le(\|u_n\|_p^{p-1}+\|u_n\|_p^{p-1})\|u_n-u\|_p\to 0.
\end{align*}
Hence $J$ satisfies $(PS)$.
\end{proof}

Similarly to what done in Section~\ref{Sect3}, we can now gather all the previous results to prove Theorem~\ref{mainthm2}.
\begin{proof}[Proof of Theorem~\ref{mainthm2}]
    The proof follows by Lemmas~\ref{condgeome2}--\ref{PS2} and Theorem~\ref{Rabinowitz86}.
\end{proof}


\appendix

\section{Spectral Theory for mixed local-nonlocal operators}\label{SectA}
In this appendix we present a proof for the spectral decomposition result used in Section~\ref{Sect4}, which is Proposition~\ref{thm:spc}. This proof of this result is quite standard in the literature, we refer for example to~\cite[Theorem 1]{Evans} and~\cite[Theorem 9.1]{Brezis}. For the sake of completeness, we present it here in a more general abstract setting.

Let $H$ be a separable (real) Hilbert space with scalar product $(\,\cdot,\,\cdot\,)_H$ and corresponding norm $\|\,\cdot\,\|_H$, and let $V\subset H$ be a Hilbert subspace with scalar product $(\,\cdot,\,\cdot\,)_V$ and corresponding norm $\|\,\cdot\,\|_V$. Assume that the embedding $V\subset H$ is dense, continuous, and compact. We identify $H$ with its dual $H'$, so the embedding $V\subset H$ induces the embedding $H\subset V'$, defined as
\begin{equation}
\langle h,v\rangle_{V'\times V}:=(h,v)_H\quad\text{for all $h\in H$ and $v\in V$}.
\end{equation}
Notice that the embedding $H\subset V'$ is dense and continuous by our assumption on $V$ and $H$.

Let $\mathcal B\colon V\times V\to \R$ be a symmetric bilinear form on $V$ satisfying:
\begin{itemize}
\item $\mathcal B$ is continuous, i.e., there exists a constant $K>0$ such that
\begin{equation}
|\mathcal B(v,w)|\le K\|v\|_V\|w\|_V\quad\text{for all $v,w\in V$};
\end{equation}
\item there exists $\gamma \ge 0$ and $\beta>0$ such that
\begin{equation}
\mathcal B(v,v)+\gamma \|v\|_H^2\ge \beta\|v\|_V^2, \quad\text{for all $u\in V$}.
\end{equation}
\end{itemize}

\begin{definition}
We say that two vectors $v,w\in V$ are $\mathcal B$-orthogonal if $\mathcal B(v,w)=0$.
\end{definition}

We associate to the bilinear form $\mathcal B$ a linear and continuous map $\mathcal L\colon V\to V'$ as
\begin{equation}
\langle\mathcal Lv,w\rangle_{V'\times V}=\mathcal B(v,w)\quad\text{for all $v,w\in V$}.
\end{equation}

\begin{definition}
We say that $\lambda\in\R$ is an eigenvalue of $\mathcal L$ in $V$ if there exists a vector $v\in V\setminus\{0\}$ such that
\begin{equation}
\mathcal Lv=\lambda v\quad\text{in $V'$},
\end{equation}
or equivalently
\begin{equation}
\mathcal B(v,w)=\lambda(v,w)_H\quad\text{for all $v,w\in V$}.
\end{equation}
The vector $v\in V\setminus\{0\}$ is called eigenvector corresponding to the eigenvalue $\lambda$.
\end{definition}

\begin{definition}
Let $\lambda\in\R$ be an eigenvalue of $\mathcal L$. We say that $\lambda$ has finite multiplicity if
\begin{equation}
\{v\in V\,:\,\mathcal Lv=\lambda v\}
\end{equation}
is a finite dimensional linear subspace of $V$.

\end{definition}

\begin{proposition}\label{thm:spc}
Under the previous assumptions, there exists an increasing sequence $(\lambda_k)_k$ of eigenvalues of $\mathcal L$ satisfying
\begin{equation}
-\gamma < \lambda_1\le \lambda_2\le\cdots\le \lambda_k\to \infty\quad\text{as $k\to\infty$}.
\end{equation}
Moreover, for all $k\in \mathbb N$ the eigenvalue $\lambda_k$ has finite multiplicity and there exists a sequence of eigenvectors $(e_k)_k\subset V$ corresponding to $(\lambda_k)_k$ satisfying
\begin{itemize}
\item[(a)] $(e_k)_k$ is an orthonormal basis of $H$;
\item[(b)] $e_k$ and $e_j$ are $\mathcal B$-orthogonal for all $k,j\in\mathbb N$ with $k\neq j$.
\end{itemize}
Finally, if we define $\mathbb P_1:=V$ and
\begin{equation}
\mathbb P_k:=\{v\in V\,:\,(v,e_j)_H=0\text{ for all $j=1,\dots,k-1$}\}\quad\text{for all $k\ge 2$},
\end{equation}
then for all $k\in\mathbb N$ we can characterize the eigenvalue $\lambda_k$ as
\begin{equation}\label{eq:lambdak}
\lambda_k:=\min_{u\in \mathbb P_k\setminus\{0\}}\frac{\mathcal B(u,u)}{\|u\|_H^2},
\end{equation}
and the eigenvector $e_k$ corresponding to the eigenvalue $\lambda_k$ realizes the minimum.
\end{proposition}

\begin{proof}
We define the bilinear form $\mathcal B^\gamma \colon V\times V\to \R$ as
\begin{equation}
\mathcal B^\gamma (v,w):=\mathcal B(v,w)+\gamma (v,w)_H\quad\text{for all $v,w\in V$},
\end{equation}
and we consider the associated linear and continuous map $\mathcal L^\gamma\colon V\to V'$ (notice that $\mathcal L^\gamma$ and $\mathcal L$ are related by the relation $\mathcal L^\gamma v=\mathcal Lv+\gamma v$ for all $v\in V$). The bilinear form $\mathcal B^\gamma $ is symmetric, continuous, and it satisfies
\begin{equation}\label{eq:scal}
\mathcal B^\gamma (v,v)\ge \beta\|v\|_V^2\quad\text{for all $v\in V$}.
\end{equation}
Therefore, by Lax-Milgram theorem $\mathcal L^\gamma \colon V\to V'$ is invertible and we can consider the inverse operator $(\mathcal L^\gamma)^{-1}\colon V'\to V$ which is still linear and continuous.

Since $\mathcal L^\gamma$ is invertible and the embedding $V\subset H$ is dense, we derive that $\lambda\in\R$ is an eigenvalue of $\mathcal L$ in $V$ with eigenvector $v\in V\setminus\{0\}$ if and only $\frac{1}{\lambda+\gamma}$ is an eigenvalue of $R^\gamma:=(\mathcal L^\gamma)^{-1}$ in $H$ with eigenvector $v\in H\setminus\{0\}$. The operator $R^\gamma\colon H\to H$ is linear, continuous, and compact, since the embedding $V\subset H$ is compact. Moreover $R^\gamma$ is injective and self-adjoint, being $\mathcal B^\gamma $ symmetric, and
\begin{equation}
(R^\gamma h,h)_H=\mathcal B^\gamma (R^\gamma h,R^\gamma h)\ge 0\quad\text{for all $h\in H$}.
\end{equation}

Therefore, by~\cite[Theorem 6.9 and Theorem 6.11]{Brezis} there exists a decreasing sequence $(\mu_k)_k$ of eigenvalues of $R^\gamma$ with $\mu_k>0$ and $\mu_k\to 0$ as $k\to\infty$. Moreover, every $\mu_k$ has finite multiplicity and there exists a orthonormal basis $(e_k)_k$ of $H$ given by eigenvectors of $R^\gamma$ associated to $\mu_k$. Hence, if we consider
\begin{equation}
\lambda_k:=\frac{1}{\mu_k}-\gamma\quad\text{for all $k\in\mathbb N$},
\end{equation}
we get that $(\lambda_k)_k$ is an increasing sequence of eigenvalues of $\mathcal L$ in $V$, with $\lambda_k\to\infty$ as $k\to\infty$ and such that every $\lambda_k$ has finite multiplicity. Moreover, for all $k\in\mathbb N$ the vector $e_k\in V\setminus\{0\}$ is an eigenvector for $\lambda_k$ and for all $k,j\in\mathbb N$ with $k\neq j$
\begin{equation}
\mathcal B(e_k,e_j)=\lambda_k(e_k,e_j)_H=0,
\end{equation}
i.e., $e_k$ and $e_j$ are $\mathcal B$-orthogonal.

Let us prove~\eqref{eq:lambdak}. By~\eqref{eq:scal} the bilinear form $\mathcal B^\gamma $ is a scalar product on $V$ equivalent to the standard one. Moreover, for all $k,j\in\mathbb N$ with $k\neq j$ we have
\begin{equation}\label{eq:ort}
\mathcal B^\gamma (e_k,e_k)=\frac{1}{\mu_k}\|e_k\|_H^2=\frac{1}{\mu_k},\quad\mathcal B^\gamma (e_k,e_j)=\frac{1}{\mu_k}(e_k,e_j)_H=0.
\end{equation}
This implies that $(\sqrt{\mu_k}{e_k})_k$ is an orthonormal basis in $V$ with respect to the scalar product $\mathcal B^\gamma $. In particular, we obtain
\begin{align*}
&\|v\|_H^2=\sum_{k=1}^\infty| (v,e_k)_H|^2,\quad\mathcal B^\gamma (v,v)=\sum_{k=1}^\infty |\mathcal B^\gamma (v,\sqrt{\mu_k}e_k)|^2=\sum_{k=1}^\infty\frac{1}{\mu_k}|(v,e_k)_H|^2.
\end{align*}
Hence, we get~\eqref{eq:lambdak} for all $k\in\N$ by exploiting the definition of $\mathbb P_k$.
\end{proof}

\begin{remark}\label{rem:dec}
Let us define
\begin{equation}
H_1:=\{0\},\quad H_k:=\textrm{span}\{e_1,\dots,e_{k-1}\}\subset V\text{ for all $k\ge 2$}.
\end{equation}
By~\eqref{eq:ort} the closed linear subspaces $H_k$ and $\mathbb P_k$ are $\mathcal B^\gamma $-orthogonal. In particular, $V$ is decomposable as direct sum $V = H_k \oplus \mathbb P_k$ for all $k\in\N$.
\end{remark}

\section*{Statements and Declarations}
\noindent\textbf{Conflict of interest}  The authors declare that they have  no conflict of interest.

\bigskip

\noindent\textbf{Acknowledgements.} The authors are members and acknowledge the support of {\it Gruppo Nazionale per l’Analisi Matematica, la Probabilità e le loro Applicazioni} (GNAMPA) of {\it Istituto Nazionale di Alta Matematica} (INdAM).

N.~Cangiotti acknowledges the support of the MIUR - PRIN 2017 project ``From Models to Decisions" (Prot. N. 201743F9YE).

M.~Caponi and A.~Maione acknowledge the support of the DFG SPP 2256 project ``Variational Methods for Predicting Complex Phenomena in Engineering Structures and Materials''.

M.~Caponi, A.~Maione and E.~Vitillaro are also supported by the INdAM - GNAMPA Project ``Equazioni differenziali alle derivate parziali di tipo misto o dipendenti da campi di vettori'' (Project number CUP\_E53C22001930001).

M.~Caponi acknowledges also the support of the project STAR PLUS 2020 - Linea 1 (21-UNINA-EPIG-172) ``New perspectives in the Variational modeling of Continuum Mechanics'', and of the MIUR - PRIN 2017 project ``Variational Methods for Stationary and Evolution Problems with Singularities and Interfaces'' (Prot. N. 2017BTM7SN).

E.~Vitillaro is supported by  ``Progetti Equazioni delle onde con condizioni iperboliche ed acustiche al bordo,  finanziati  con  i Fondi  Ricerca  di Base 2017--2022, della Universit\`a degli Studi di Perugia''.

The authors wish to thank Dimitri Mugnai for many useful discussions.



\end{document}